\documentclass[a4paper,11pt]{article}

\usepackage{MyStyle}

\usepackage{hyperref}

\usepackage{planarforest}
\newcommand{\forestA}{
\tikz[planar forest ] {

\draw (-1.0,1.0) edge[-] (0.0,1.0);
\draw (0.0,1.0) edge[-] (1.5,1.0);
\draw (-1.0,1.0) arc (90:270:0.5);
\draw (1.5,1.0) arc (90:-90:0.5);
\draw (-1.0,0.0) edge[-] (1.5,0.0);
\node [draw=none] at (0.0, 0.0) { } 
child {node [b] at (-1.0, 1.0) {  } edge from parent[draw=none] 
}
child {node [b] at (0.0, 1.0) {  } edge from parent[draw=none] 
child {node [b] at (0.0, 1.0) {  }  
}
}
child {node [b] at (1.5, 1.0) {  } edge from parent[draw=none] 
child {node [b] at (-0.5, 1.0) {  }  
}
child {node [b] at (0.5, 1.0) {  }  
}
}
;
}}

\newcommand{\forestB}{
\tikz[planar forest ] {

\draw (-1.5,1.0) edge[-] (0.0,1.0);
\draw (0.0,1.0) edge[-] (1.0,1.0);
\draw (-1.5,1.0) arc (90:270:0.5);
\draw (1.0,1.0) arc (90:-90:0.5);
\draw (-1.5,0.0) edge[-] (1.0,0.0);
\node [draw=none] at (0.0, 0.0) { } 
child {node [b] at (-1.5, 1.0) {  } edge from parent[draw=none] 
child {node [b] at (0.0, 1.0) {  }  
}
}
child {node [b] at (0.0, 1.0) {  } edge from parent[draw=none] 
child {node [b] at (-0.5, 1.0) {  }  
}
child {node [b] at (0.5, 1.0) {  }  
}
}
child {node [b] at (1.0, 1.0) {  } edge from parent[draw=none] 
}
;
}}

\newcommand{\forestC}{
\tikz[planar forest ] {

\draw (-1.0,1.0) edge[-] (0.0,1.0);
\draw (0.0,1.0) edge[-] (1.0,1.0);
\draw (-1.0,1.0) arc (90:270:0.5);
\draw (1.0,1.0) arc (90:-90:0.5);
\draw (-1.0,0.0) edge[-] (1.0,0.0);
\node [draw=none] at (0.0, 0.0) { } 
child {node [b] at (-1.0, 1.0) {  } edge from parent[draw=none] 
child {node [b] at (-0.5, 1.0) {  }  
}
child {node [b] at (0.5, 1.0) {  }  
}
}
child {node [b] at (0.0, 1.0) {  } edge from parent[draw=none] 
}
child {node [b] at (1.0, 1.0) {  } edge from parent[draw=none] 
child {node [b] at (0.0, 1.0) {  }  
}
}
;
}}

\newcommand{\forestD}{
\tikz[planar forest ] {

\node [b] at (0.0, 0.0) {  } 
;
}}

\newcommand{\forestE}{
\tikz[planar forest ] {

\node [b] at (0.0, 0.0) {  } 
child {node [b] at (0.0, 1.0) {  }  
}
;
}}

\newcommand{\forestF}{
\tikz[planar forest ] {

\node [b] at (0.0, 0.0) {  } 
child {node [b] at (-0.5, 1.0) {  }  
}
child {node [b] at (0.5, 1.0) {  }  
}
;
}}

\newcommand{\forestG}{
\tikz[planar forest ] {

\node [b] at (0.0, 0.0) {  } 
child {node [b] at (0.0, 1.0) {  }  
}
;
}}

\newcommand{\forestH}{
\tikz[planar forest ] {

\node [b] at (0.0, 0.0) {  } 
child {node [b] at (-0.5, 1.0) {  }  
}
child {node [b] at (0.5, 1.0) {  }  
}
;
}}

\newcommand{\forestI}{
\tikz[planar forest ] {

\node [b] at (0.0, 0.0) {  } 
;
}}

\newcommand{\forestJ}{
\tikz[planar forest ] {

\node [b] at (0.0, 0.0) {  } 
child {node [b] at (-0.5, 1.0) {  }  
}
child {node [b] at (0.5, 1.0) {  }  
}
;
}}

\newcommand{\forestK}{
\tikz[planar forest ] {

\node [b] at (0.0, 0.0) {  } 
;
}}

\newcommand{\forestL}{
\tikz[planar forest ] {

\node [b] at (0.0, 0.0) {  } 
child {node [b] at (0.0, 1.0) {  }  
}
;
}}

\newcommand{\forestM}{
\tikz[planar forest ] {

\node [b] at (0.0, 0.0) {  } 
;
}}

\newcommand{\forestN}{
\tikz[planar forest ] {

\node [b] at (0.0, 0.0) {  } 
child {node [b] at (0.0, 1.0) {  }  
}
;
}}

\newcommand{\forestO}{
\tikz[planar forest ] {

\draw (0.0,1.0) arc (90:270:0.5);
\draw (0.0,1.0) arc (90:-90:0.5);
\node [draw=none] at (0.0, 0.0) { } 
child {node [b] at (0.0, 1.0) {  } edge from parent[draw=none] 
}
;
\node [b] at (1.1, 0.0) {  } 
;
}}

\newcommand{\forestP}{
\tikz[planar forest ] {

\node [b] at (0.0, 0.0) {  } 
child {node [b] at (0.0, 1.0) {  }  
child {node [b] at (0.0, 1.0) {  }  
}
}
;
}}

\newcommand{\forestQ}{
\tikz[planar forest ] {

\node [b] at (0.0, 0.0) {  } 
child {node [b] at (-0.5, 1.0) {  }  
}
child {node [b] at (0.5, 1.0) {  }  
}
;
}}

\newcommand{\forestR}{
\tikz[planar forest ] {

\draw (0.0,1.0) arc (90:270:0.5);
\draw (0.0,1.0) arc (90:-90:0.5);
\node [draw=none] at (0.0, 0.0) { } 
child {node [b] at (0.0, 1.0) {  } edge from parent[draw=none] 
}
;
\node [b] at (1.1, 0.0) {  } 
child {node [b] at (0.0, 1.0) {  }  
}
;
}}

\newcommand{\forestS}{
\tikz[planar forest ] {

\draw (0.0,1.0) arc (90:270:0.5);
\draw (0.0,1.0) arc (90:-90:0.5);
\node [draw=none] at (0.0, 0.0) { } 
child {node [b] at (0.0, 1.0) {  } edge from parent[draw=none] 
child {node [b] at (0.0, 1.0) {  }  
}
}
;
\node [b] at (1.1, 0.0) {  } 
;
}}

\newcommand{\forestT}{
\tikz[planar forest ] {

\draw (-0.5,1.0) edge[-] (0.5,1.0);
\draw (-0.5,1.0) arc (90:270:0.5);
\draw (0.5,1.0) arc (90:-90:0.5);
\draw (-0.5,0.0) edge[-] (0.5,0.0);
\node [draw=none] at (0.0, 0.0) { } 
child {node [b] at (-0.5, 1.0) {  } edge from parent[draw=none] 
}
child {node [b] at (0.5, 1.0) {  } edge from parent[draw=none] 
}
;
\node [b] at (1.6, 0.0) {  } 
;
}}

\newcommand{\forestU}{
\tikz[planar forest ] {

\draw (-4.440892098500626e-16,1.0) arc (90:270:0.5);
\draw (-4.440892098500626e-16,1.0) arc (90:-90:0.5);
\node [draw=none] at (-4.440892098500626e-16, 0.0) { } 
child {node [b] at (0.0, 1.0) {  } edge from parent[draw=none] 
}
;
\draw (1.2,1.0) arc (90:270:0.5);
\draw (1.2,1.0) arc (90:-90:0.5);
\node [draw=none] at (1.2, 0.0) { } 
child {node [b] at (0.0, 1.0) {  } edge from parent[draw=none] 
}
;
\node [b] at (2.3, 0.0) {  } 
;
}}

\newcommand{\forestV}{
\tikz[planar forest ] {

\node [b] at (0.0, 0.0) {  } 
child {node [b] at (0.0, 1.0) {  }  
child {node [b] at (0.0, 1.0) {  }  
child {node [b] at (0.0, 1.0) {  }  
}
}
}
;
}}

\newcommand{\forestW}{
\tikz[planar forest ] {

\node [b] at (0.0, 0.0) {  } 
child {node [b] at (0.0, 1.0) {  }  
child {node [b] at (-0.5, 1.0) {  }  
}
child {node [b] at (0.5, 1.0) {  }  
}
}
;
}}

\newcommand{\forestX}{
\tikz[planar forest ] {

\node [b] at (0.0, 0.0) {  } 
child {node [b] at (-0.5, 1.0) {  }  
}
child {node [b] at (0.5, 1.0) {  }  
child {node [b] at (0.0, 1.0) {  }  
}
}
;
}}

\newcommand{\forestY}{
\tikz[planar forest ] {

\node [b] at (0.0, 0.0) {  } 
child {node [b] at (-1.0, 1.0) {  }  
}
child {node [b] at (0.0, 1.0) {  }  
}
child {node [b] at (1.0, 1.0) {  }  
}
;
}}

\newcommand{\forestAB}{
\tikz[planar forest ] {

\draw (0.0,1.0) arc (90:270:0.5);
\draw (0.0,1.0) arc (90:-90:0.5);
\node [draw=none] at (0.0, 0.0) { } 
child {node [b] at (0.0, 1.0) {  } edge from parent[draw=none] 
}
;
\node [b] at (1.1, 0.0) {  } 
child {node [b] at (0.0, 1.0) {  }  
child {node [b] at (0.0, 1.0) {  }  
}
}
;
}}

\newcommand{\forestBB}{
\tikz[planar forest ] {

\draw (0.0,1.0) arc (90:270:0.5);
\draw (0.0,1.0) arc (90:-90:0.5);
\node [draw=none] at (0.0, 0.0) { } 
child {node [b] at (0.0, 1.0) {  } edge from parent[draw=none] 
}
;
\node [b] at (1.5, 0.0) {  } 
child {node [b] at (-0.5, 1.0) {  }  
}
child {node [b] at (0.5, 1.0) {  }  
}
;
}}

\newcommand{\forestCB}{
\tikz[planar forest ] {

\draw (0.0,1.0) arc (90:270:0.5);
\draw (0.0,1.0) arc (90:-90:0.5);
\node [draw=none] at (0.0, 0.0) { } 
child {node [b] at (0.0, 1.0) {  } edge from parent[draw=none] 
child {node [b] at (0.0, 1.0) {  }  
}
}
;
\node [b] at (1.1, 0.0) {  } 
child {node [b] at (0.0, 1.0) {  }  
}
;
}}

\newcommand{\forestDB}{
\tikz[planar forest ] {

\draw (-0.5,1.0) edge[-] (0.5,1.0);
\draw (-0.5,1.0) arc (90:270:0.5);
\draw (0.5,1.0) arc (90:-90:0.5);
\draw (-0.5,0.0) edge[-] (0.5,0.0);
\node [draw=none] at (0.0, 0.0) { } 
child {node [b] at (-0.5, 1.0) {  } edge from parent[draw=none] 
}
child {node [b] at (0.5, 1.0) {  } edge from parent[draw=none] 
}
;
\node [b] at (1.6, 0.0) {  } 
child {node [b] at (0.0, 1.0) {  }  
}
;
}}

\newcommand{\forestEB}{
\tikz[planar forest ] {

\draw (-4.440892098500626e-16,1.0) arc (90:270:0.5);
\draw (-4.440892098500626e-16,1.0) arc (90:-90:0.5);
\node [draw=none] at (-4.440892098500626e-16, 0.0) { } 
child {node [b] at (0.0, 1.0) {  } edge from parent[draw=none] 
}
;
\draw (1.2,1.0) arc (90:270:0.5);
\draw (1.2,1.0) arc (90:-90:0.5);
\node [draw=none] at (1.2, 0.0) { } 
child {node [b] at (0.0, 1.0) {  } edge from parent[draw=none] 
}
;
\node [b] at (2.3, 0.0) {  } 
child {node [b] at (0.0, 1.0) {  }  
}
;
}}

\newcommand{\forestFB}{
\tikz[planar forest ] {

\draw (0.0,1.0) arc (90:270:0.5);
\draw (0.0,1.0) arc (90:-90:0.5);
\node [draw=none] at (0.0, 0.0) { } 
child {node [b] at (0.0, 1.0) {  } edge from parent[draw=none] 
child {node [b] at (0.0, 1.0) {  }  
child {node [b] at (0.0, 1.0) {  }  
}
}
}
;
\node [b] at (1.1, 0.0) {  } 
;
}}

\newcommand{\forestGB}{
\tikz[planar forest ] {

\draw (0.0,1.0) arc (90:270:0.5);
\draw (0.0,1.0) arc (90:-90:0.5);
\node [draw=none] at (0.0, 0.0) { } 
child {node [b] at (0.0, 1.0) {  } edge from parent[draw=none] 
child {node [b] at (-0.5, 1.0) {  }  
}
child {node [b] at (0.5, 1.0) {  }  
}
}
;
\node [b] at (1.1, 0.0) {  } 
;
}}

\newcommand{\forestHB}{
\tikz[planar forest ] {

\draw (-0.5,1.0) edge[-] (0.5,1.0);
\draw (-0.5,1.0) arc (90:270:0.5);
\draw (0.5,1.0) arc (90:-90:0.5);
\draw (-0.5,0.0) edge[-] (0.5,0.0);
\node [draw=none] at (0.0, 0.0) { } 
child {node [b] at (-0.5, 1.0) {  } edge from parent[draw=none] 
}
child {node [b] at (0.5, 1.0) {  } edge from parent[draw=none] 
child {node [b] at (0.0, 1.0) {  }  
}
}
;
\node [b] at (1.6, 0.0) {  } 
;
}}

\newcommand{\forestIB}{
\tikz[planar forest ] {

\draw (-1.0,1.0) edge[-] (0.0,1.0);
\draw (0.0,1.0) edge[-] (1.0,1.0);
\draw (-1.0,1.0) arc (90:270:0.5);
\draw (1.0,1.0) arc (90:-90:0.5);
\draw (-1.0,0.0) edge[-] (1.0,0.0);
\node [draw=none] at (0.0, 0.0) { } 
child {node [b] at (-1.0, 1.0) {  } edge from parent[draw=none] 
}
child {node [b] at (0.0, 1.0) {  } edge from parent[draw=none] 
}
child {node [b] at (1.0, 1.0) {  } edge from parent[draw=none] 
}
;
\node [b] at (2.1, 0.0) {  } 
;
}}

\newcommand{\forestJB}{
\tikz[planar forest ] {

\draw (-4.440892098500626e-16,1.0) arc (90:270:0.5);
\draw (-4.440892098500626e-16,1.0) arc (90:-90:0.5);
\node [draw=none] at (-4.440892098500626e-16, 0.0) { } 
child {node [b] at (0.0, 1.0) {  } edge from parent[draw=none] 
}
;
\draw (1.2,1.0) arc (90:270:0.5);
\draw (1.2,1.0) arc (90:-90:0.5);
\node [draw=none] at (1.2, 0.0) { } 
child {node [b] at (0.0, 1.0) {  } edge from parent[draw=none] 
child {node [b] at (0.0, 1.0) {  }  
}
}
;
\node [b] at (2.3, 0.0) {  } 
;
}}

\newcommand{\forestKB}{
\tikz[planar forest ] {

\draw (4.440892098500626e-16,1.0) arc (90:270:0.5);
\draw (4.440892098500626e-16,1.0) arc (90:-90:0.5);
\node [draw=none] at (4.440892098500626e-16, 0.0) { } 
child {node [b] at (0.0, 1.0) {  } edge from parent[draw=none] 
}
;
\draw (1.2,1.0) edge[-] (2.2,1.0);
\draw (1.2,1.0) arc (90:270:0.5);
\draw (2.2,1.0) arc (90:-90:0.5);
\draw (1.2,0.0) edge[-] (2.2,0.0);
\node [draw=none] at (1.7, 0.0) { } 
child {node [b] at (-0.5, 1.0) {  } edge from parent[draw=none] 
}
child {node [b] at (0.5000000000000002, 1.0) {  } edge from parent[draw=none] 
}
;
\node [b] at (3.3, 0.0) {  } 
;
}}

\newcommand{\forestLB}{
\tikz[planar forest ] {

\draw (0.0,1.0) arc (90:270:0.5);
\draw (0.0,1.0) arc (90:-90:0.5);
\node [draw=none] at (0.0, 0.0) { } 
child {node [b] at (0.0, 1.0) {  } edge from parent[draw=none] 
}
;
\draw (1.2,1.0) arc (90:270:0.5);
\draw (1.2,1.0) arc (90:-90:0.5);
\node [draw=none] at (1.2, 0.0) { } 
child {node [b] at (0.0, 1.0) {  } edge from parent[draw=none] 
}
;
\draw (2.4000000000000004,1.0) arc (90:270:0.5);
\draw (2.4000000000000004,1.0) arc (90:-90:0.5);
\node [draw=none] at (2.4000000000000004, 0.0) { } 
child {node [b] at (0.0, 1.0) {  } edge from parent[draw=none] 
}
;
\node [b] at (3.5000000000000004, 0.0) {  } 
;
}}

\newcommand{\forestMB}{
\tikz[planar forest ] {

\draw (-0.5,1.0) edge[-] (0.5,1.0);
\draw (-0.5,1.0) arc (90:270:0.5);
\draw (0.5,1.0) arc (90:-90:0.5);
\draw (-0.5,0.0) edge[-] (0.5,0.0);
\node [draw=none] at (0.0, 0.0) { } 
child {node [b] at (-0.5, 1.0) {  } edge from parent[draw=none] 
}
child {node [b] at (0.5, 1.0) {  } edge from parent[draw=none] 
}
;
\node [b] at (1.6, 0.0) {  } 
child {node [b] at (0.0, 1.0) {  }  
}
;
}}

\newcommand{\forestNB}{
\tikz[planar forest ] {

\draw (0.0,1.0) arc (90:270:0.5);
\draw (0.0,1.0) arc (90:-90:0.5);
\node [draw=none] at (0.0, 0.0) { } 
child {node [b] at (0.0, 1.0) {  } edge from parent[draw=none] 
child {node [b] at (0.0, 1.0) {  }  
}
}
;
\node [b] at (1.1, 0.0) {  } 
child {node [b] at (0.0, 1.0) {  }  
}
;
}}

\newcommand{\forestOB}{
\tikz[planar forest ] {

\draw (-0.5,1.0) edge[-] (0.5,1.0);
\draw (-0.5,1.0) arc (90:270:0.5);
\draw (0.5,1.0) arc (90:-90:0.5);
\draw (-0.5,0.0) edge[-] (0.5,0.0);
\node [draw=none] at (0.0, 0.0) { } 
child {node [b] at (-0.5, 1.0) {  } edge from parent[draw=none] 
}
child {node [b] at (0.5, 1.0) {  } edge from parent[draw=none] 
}
;
\draw (1.7,1.0) arc (90:270:0.5);
\draw (1.7,1.0) arc (90:-90:0.5);
\node [draw=none] at (1.7, 0.0) { } 
child {node [b] at (0.0, 1.0) {  } edge from parent[draw=none] 
child {node [b] at (0.0, 1.0) {  }  
}
}
;
\node [b] at (3.2, 0.0) {  } 
child {node [b] at (-0.5, 1.0) {  }  
}
child {node [b] at (0.5, 1.0) {  }  
child {node [b] at (0.0, 1.0) {  }  
}
}
;
}}

\newcommand{\forestPB}{
\tikz[planar forest ] {

\draw (-0.5000000000000004,1.0) edge[-] (0.49999999999999956,1.0);
\draw (-0.5000000000000004,1.0) arc (90:270:0.5);
\draw (0.49999999999999956,1.0) arc (90:-90:0.5);
\draw (-0.5000000000000004,0.0) edge[-] (0.49999999999999956,0.0);
\node [draw=none] at (-4.440892098500626e-16, 0.0) { } 
child {node [b] at (-0.5, 1.0) {  } edge from parent[draw=none] 
}
child {node [b] at (0.5, 1.0) {  } edge from parent[draw=none] 
}
;
\draw (1.7,1.0) arc (90:270:0.5);
\draw (1.7,1.0) arc (90:-90:0.5);
\node [draw=none] at (1.7, 0.0) { } 
child {node [b] at (0.0, 1.0) {  } edge from parent[draw=none] 
child {node [b] at (0.0, 1.0) {  }  
}
}
;
\node [b] at (2.8, 0.0) {  } 
child {node [b] at (0.0, 1.0) {  }  
child {node [b] at (0.0, 1.0) {  }  
child {node [b] at (0.0, 1.0) {  }  
}
}
}
;
}}

\newcommand{\forestQB}{
\tikz[planar forest ] {

\draw (-0.5000000000000004,1.0) edge[-] (0.49999999999999956,1.0);
\draw (-0.5000000000000004,1.0) arc (90:270:0.5);
\draw (0.49999999999999956,1.0) arc (90:-90:0.5);
\draw (-0.5000000000000004,0.0) edge[-] (0.49999999999999956,0.0);
\node [draw=none] at (-4.440892098500626e-16, 0.0) { } 
child {node [b] at (-0.5, 1.0) {  } edge from parent[draw=none] 
}
child {node [b] at (0.5, 1.0) {  } edge from parent[draw=none] 
}
;
\draw (1.7,1.0) arc (90:270:0.5);
\draw (1.7,1.0) arc (90:-90:0.5);
\node [draw=none] at (1.7, 0.0) { } 
child {node [b] at (0.0, 1.0) {  } edge from parent[draw=none] 
child {node [b] at (-0.5, 1.0) {  }  
}
child {node [b] at (0.5000000000000002, 1.0) {  }  
child {node [b] at (0.0, 1.0) {  }  
}
}
}
;
\node [b] at (2.8, 0.0) {  } 
child {node [b] at (0.0, 1.0) {  }  
}
;
}}

\newcommand{\forestRB}{
\tikz[planar forest ] {

\draw (-0.5000000000000004,1.0) edge[-] (0.49999999999999956,1.0);
\draw (-0.5000000000000004,1.0) arc (90:270:0.5);
\draw (0.49999999999999956,1.0) arc (90:-90:0.5);
\draw (-0.5000000000000004,0.0) edge[-] (0.49999999999999956,0.0);
\node [draw=none] at (-4.440892098500626e-16, 0.0) { } 
child {node [b] at (-0.5, 1.0) {  } edge from parent[draw=none] 
}
child {node [b] at (0.5, 1.0) {  } edge from parent[draw=none] 
}
;
\draw (1.7,1.0) arc (90:270:0.5);
\draw (1.7,1.0) arc (90:-90:0.5);
\node [draw=none] at (1.7, 0.0) { } 
child {node [b] at (0.0, 1.0) {  } edge from parent[draw=none] 
child {node [b] at (0.0, 1.0) {  }  
child {node [b] at (0.0, 1.0) {  }  
child {node [b] at (0.0, 1.0) {  }  
}
}
}
}
;
\node [b] at (2.8, 0.0) {  } 
child {node [b] at (0.0, 1.0) {  }  
}
;
}}

\newcommand{\forestSB}{
\tikz[planar forest ] {

\draw (-0.5,1.0) edge[-] (0.5,1.0);
\draw (-0.5,1.0) arc (90:270:0.5);
\draw (0.5,1.0) arc (90:-90:0.5);
\draw (-0.5,0.0) edge[-] (0.5,0.0);
\node [draw=none] at (0.0, 0.0) { } 
child {node [b] at (-0.5, 1.0) {  } edge from parent[draw=none] 
}
child {node [b] at (0.5, 1.0) {  } edge from parent[draw=none] 
}
;
\node [b] at (1.6, 0.0) {  } 
child {node [b] at (0.0, 1.0) {  }  
}
;
}}

\newcommand{\forestTB}{
\tikz[planar forest ] {

\draw (-0.5,1.0) edge[-] (0.5000000000000002,1.0);
\draw (-0.5,1.0) arc (90:270:0.5);
\draw (0.5000000000000002,1.0) arc (90:-90:0.5);
\draw (-0.5,0.0) edge[-] (0.5000000000000002,0.0);
\node [draw=none] at (0.0, 0.0) { } 
child {node [b] at (-0.5, 1.0) {  } edge from parent[draw=none] 
}
child {node [b] at (0.5000000000000002, 1.0) {  } edge from parent[draw=none] 
}
;
\draw (1.7,1.0) arc (90:270:0.5);
\draw (1.7,1.0) arc (90:-90:0.5);
\node [draw=none] at (1.7, 0.0) { } 
child {node [b] at (0.0, 1.0) {  } edge from parent[draw=none] 
child {node [b] at (0.0, 1.0) {  }  
}
}
;
}}

\newcommand{\forestUB}{
\tikz[planar forest ] {

\draw (-0.5,1.0) edge[-] (0.5,1.0);
\draw (-0.5,1.0) arc (90:270:0.5);
\draw (0.5,1.0) arc (90:-90:0.5);
\draw (-0.5,0.0) edge[-] (0.5,0.0);
\node [draw=none] at (0.0, 0.0) { } 
child {node [b] at (-0.5, 1.0) {  } edge from parent[draw=none] 
}
child {node [b] at (0.5, 1.0) {  } edge from parent[draw=none] 
}
;
\draw (1.7000000000000002,1.0) edge[-] (2.7,1.0);
\draw (1.7000000000000002,1.0) arc (90:270:0.5);
\draw (2.7,1.0) arc (90:-90:0.5);
\draw (1.7000000000000002,0.0) edge[-] (2.7,0.0);
\node [draw=none] at (2.2, 0.0) { } 
child {node [b] at (-0.5, 1.0) {  } edge from parent[draw=none] 
}
child {node [b] at (0.5, 1.0) {  } edge from parent[draw=none] 
}
;
}}

\newcommand{\forestVB}{
\tikz[planar forest ] {

\draw (-0.5,1.0) edge[-] (0.5,1.0);
\draw (-0.5,1.0) arc (90:270:0.5);
\draw (0.5,1.0) arc (90:-90:0.5);
\draw (-0.5,0.0) edge[-] (0.5,0.0);
\node [draw=none] at (0.0, 0.0) { } 
child {node [b] at (-0.5, 1.0) {  } edge from parent[draw=none] 
}
child {node [b] at (0.5, 1.0) {  } edge from parent[draw=none] 
child {node [b] at (0.0, 1.0) {  }  
child {node [b] at (0.0, 1.0) {  }  
}
}
}
;
}}

\newcommand{\forestWB}{
\tikz[planar forest ] {

\node [l] at (0.0, 0.0) { $\times$ } 
;
}}

\newcommand{\forestXB}{
\tikz[planar forest ] {

\draw (0.0,1.0) arc (90:270:0.5);
\draw (0.0,1.0) arc (90:-90:0.5);
\node [draw=none] at (0.0, 0.0) { } 
child {node [b] at (0.0, 1.0) {  } edge from parent[draw=none] 
}
;
\node [b] at (1.5, 0.0) {  } 
child {node [b] at (-0.5, 1.0) {  }  
child {node [l] at (-0.5, 1.0) { $\times$ }  
}
child {node [b] at (0.5, 1.0) {  }  
}
}
child {node [b] at (0.5, 1.0) {  }  
}
;
}}

\newcommand{\forestYB}{
\tikz[planar forest ] {

\draw (0.0,1.0) arc (90:270:0.5);
\draw (0.0,1.0) arc (90:-90:0.5);
\node [draw=none] at (0.0, 0.0) { } 
child {node [b] at (0.0, 1.0) {  } edge from parent[draw=none] 
}
;
\draw (1.2,1.0) edge[-] (2.2,1.0);
\draw (1.2,1.0) arc (90:270:0.5);
\draw (2.2,1.0) arc (90:-90:0.5);
\draw (1.2,0.0) edge[-] (2.2,0.0);
\node [draw=none] at (1.7, 0.0) { } 
child {node [b] at (-0.5, 1.0) {  } edge from parent[draw=none] 
child {node [b] at (0.0, 1.0) {  }  
}
}
child {node [b] at (0.5000000000000002, 1.0) {  } edge from parent[draw=none] 
child {node [b] at (0.0, 1.0) {  }  
}
}
;
}}

\newcommand{\forestAC}{
\tikz[planar forest ] {

\draw (0.0,1.0) arc (90:270:0.5);
\draw (0.0,1.0) arc (90:-90:0.5);
\node [draw=none] at (0.0, 0.0) { } 
child {node [b] at (0.0, 1.0) {  } edge from parent[draw=none] 
}
;
\draw (1.2,1.0) arc (90:270:0.5);
\draw (1.2,1.0) arc (90:-90:0.5);
\node [draw=none] at (1.2, 0.0) { } 
child {node [b] at (0.0, 1.0) {  } edge from parent[draw=none] 
child {node [b] at (-0.5, 1.0) {  }  
}
child {node [b] at (0.5, 1.0) {  }  
}
}
;
\node [b] at (2.3, 0.0) {  } 
;
\node [b] at (3.3, 0.0) {  } 
child {node [b] at (0.0, 1.0) {  }  
}
;
\node [b] at (4.3, 0.0) {  } 
child {node [b] at (0.0, 1.0) {  }  
child {node [b] at (0.0, 1.0) {  }  
}
}
;
}}

\newcommand{\forestBC}{
\tikz[planar forest ] {

\draw (-4.440892098500626e-16,1.0) arc (90:270:0.5);
\draw (-4.440892098500626e-16,1.0) arc (90:-90:0.5);
\node [draw=none] at (-4.440892098500626e-16, 0.0) { } 
child {node [b] at (0.0, 1.0) {  } edge from parent[draw=none] 
}
;
\draw (1.2,1.0) arc (90:270:0.5);
\draw (1.2,1.0) arc (90:-90:0.5);
\node [draw=none] at (1.2, 0.0) { } 
child {node [b] at (0.0, 1.0) {  } edge from parent[draw=none] 
child {node [b] at (-0.5, 1.0) {  }  
}
child {node [b] at (0.5, 1.0) {  }  
}
}
;
\node [b] at (2.3, 0.0) {  } 
;
}}

\newcommand{\forestCC}{
\tikz[planar forest ] {

\node [b] at (0.0, 0.0) {  } 
child {node [b] at (0.0, 1.0) {  }  
}
;
\node [b] at (1.0, 0.0) {  } 
child {node [b] at (0.0, 1.0) {  }  
child {node [b] at (0.0, 1.0) {  }  
}
}
;
}}

\newcommand{\forestDC}{
\tikz[planar forest ] {

\draw (0.0,1.0) arc (90:270:0.5);
\draw (0.0,1.0) arc (90:-90:0.5);
\node [draw=none] at (0.0, 0.0) { } 
child {node [b] at (0.0, 1.0) {  } edge from parent[draw=none] 
}
;
\node [b] at (1.1, 0.0) {  } 
;
}}

\newcommand{\forestEC}{
\tikz[planar forest ] {

\node [b] at (0.0, 0.0) {  } 
child {node [b] at (0.0, 1.0) {  }  
}
;
}}

\newcommand{\forestFC}{
\tikz[planar forest ] {

\draw (0.0,1.0) arc (90:270:0.5);
\draw (0.0,1.0) arc (90:-90:0.5);
\node [draw=none] at (0.0, 0.0) { } 
child {node [b] at (0.0, 1.0) {  } edge from parent[draw=none] 
child {node [b] at (-0.5, 1.0) {  }  
}
child {node [b] at (0.5, 1.0) {  }  
}
}
;
\node [b] at (1.5, 0.0) {  } 
child {node [b] at (0.0, 1.0) {  }  
child {node [b] at (0.0, 1.0) {  }  
}
}
;
}}

\newcommand{\forestGC}{
\tikz[planar forest ] {

\draw (0.0,1.0) arc (90:270:0.5);
\draw (0.0,1.0) arc (90:-90:0.5);
\node [draw=none] at (0.0, 0.0) { } 
child {node [b] at (0.0, 1.0) {  } edge from parent[draw=none] 
child {node [b] at (0.0, 1.0) {  }  
}
}
;
\node [b] at (1.1, 0.0) {  } 
child {node [b] at (0.0, 1.0) {  }  
}
;
}}

\newcommand{\forestHC}{
\tikz[planar forest ] {

\draw (-0.5,1.0) edge[-] (0.5,1.0);
\draw (-0.5,1.0) arc (90:270:0.5);
\draw (0.5,1.0) arc (90:-90:0.5);
\draw (-0.5,0.0) edge[-] (0.5,0.0);
\node [draw=none] at (0.0, 0.0) { } 
child {node [b] at (-0.5, 1.0) {  } edge from parent[draw=none] 
}
child {node [b] at (0.5, 1.0) {  } edge from parent[draw=none] 
}
;
\node [b] at (2.0, 0.0) {  } 
child {node [b] at (-0.5, 1.0) {  }  
}
child {node [b] at (0.5, 1.0) {  }  
}
;
}}

\newcommand{\forestIC}{
\tikz[planar forest ] {

\draw (-4.440892098500626e-16,1.0) arc (90:270:0.5);
\draw (-4.440892098500626e-16,1.0) arc (90:-90:0.5);
\node [draw=none] at (-4.440892098500626e-16, 0.0) { } 
child {node [b] at (0.0, 1.0) {  } edge from parent[draw=none] 
}
;
\draw (1.2,1.0) arc (90:270:0.5);
\draw (1.2,1.0) arc (90:-90:0.5);
\node [draw=none] at (1.2, 0.0) { } 
child {node [b] at (0.0, 1.0) {  } edge from parent[draw=none] 
}
;
\node [b] at (2.3, 0.0) {  } 
;
}}

\newcommand{\forestJC}{
\tikz[planar forest ] {

\draw (0.0,1.0) arc (90:270:0.5);
\draw (0.0,1.0) arc (90:-90:0.5);
\node [draw=none] at (0.0, 0.0) { } 
child {node [b] at (0.0, 1.0) {  } edge from parent[draw=none] 
}
;
\node [b] at (1.1, 0.0) {  } 
;
\node [b] at (2.1, 0.0) {  } 
child {node [b] at (0.0, 1.0) {  }  
}
;
}}

\newcommand{\forestKC}{
\tikz[planar forest ] {

\draw (0.0,1.0) arc (90:270:0.5);
\draw (0.0,1.0) arc (90:-90:0.5);
\node [draw=none] at (0.0, 0.0) { } 
child {node [b] at (0.0, 1.0) {  } edge from parent[draw=none] 
}
;
\node [b] at (1.1, 0.0) {  } 
;
}}

\newcommand{\forestLC}{
\tikz[planar forest ] {

\node [b] at (0.0, 0.0) {  } 
child {node [b] at (0.0, 1.0) {  }  
}
;
}}

\newcommand{\forestMC}{
\tikz[planar forest ] {

\node [b] at (0.0, 0.0) {  } 
;
}}

\newcommand{\forestNC}{
\tikz[planar forest ] {

\draw (0.0,1.0) arc (90:270:0.5);
\draw (0.0,1.0) arc (90:-90:0.5);
\node [draw=none] at (0.0, 0.0) { } 
child {node [b] at (0.0, 1.0) {  } edge from parent[draw=none] 
}
;
\node [b] at (1.1, 0.0) {  } 
child {node [b] at (0.0, 1.0) {  }  
}
;
}}

\newcommand{\forestOC}{
\tikz[planar forest ] {

\draw (-4.440892098500626e-16,1.0) arc (90:270:0.5);
\draw (-4.440892098500626e-16,1.0) arc (90:-90:0.5);
\node [draw=none] at (-4.440892098500626e-16, 0.0) { } 
child {node [b] at (0.0, 1.0) {  } edge from parent[draw=none] 
}
;
\draw (1.2,1.0) arc (90:270:0.5);
\draw (1.2,1.0) arc (90:-90:0.5);
\node [draw=none] at (1.2, 0.0) { } 
child {node [b] at (0.0, 1.0) {  } edge from parent[draw=none] 
child {node [b] at (0.0, 1.0) {  }  
}
}
;
\node [b] at (2.3, 0.0) {  } 
;
\node [b] at (3.3, 0.0) {  } 
child {node [b] at (0.0, 1.0) {  }  
}
;
}}

\newcommand{\forestPC}{
\tikz[planar forest ] {

\draw (-2.220446049250313e-16,1.0) arc (90:270:0.5);
\draw (-2.220446049250313e-16,1.0) arc (90:-90:0.5);
\node [draw=none] at (-2.220446049250313e-16, 0.0) { } 
child {node [b] at (0.0, 1.0) {  } edge from parent[draw=none] 
}
;
\draw (1.2,1.0) arc (90:270:0.5);
\draw (1.2,1.0) arc (90:-90:0.5);
\node [draw=none] at (1.2, 0.0) { } 
child {node [b] at (0.0, 1.0) {  } edge from parent[draw=none] 
child {node [b] at (0.0, 1.0) {  }  
}
}
;
}}

\newcommand{\forestQC}{
\tikz[planar forest ] {

\draw (0.0,1.0) arc (90:270:0.5);
\draw (0.0,1.0) arc (90:-90:0.5);
\node [draw=none] at (0.0, 0.0) { } 
child {node [b] at (0.0, 1.0) {  } edge from parent[draw=none] 
}
;
}}

\newcommand{\forestRC}{
\tikz[planar forest ] {

\draw (0.0,1.0) arc (90:270:0.5);
\draw (0.0,1.0) arc (90:-90:0.5);
\node [draw=none] at (0.0, 0.0) { } 
child {node [b] at (0.0, 1.0) {  } edge from parent[draw=none] 
child {node [b] at (0.0, 1.0) {  }  
}
}
;
}}

\newcommand{\forestSC}{
\tikz[planar forest ] {

\draw (0.0,1.0) arc (90:270:0.5);
\draw (0.0,1.0) arc (90:-90:0.5);
\node [draw=none] at (0.0, 0.0) { } 
child {node [b] at (0.0, 1.0) {  } edge from parent[draw=none] 
child {node [b] at (0.0, 1.0) {  }  
}
}
;
}}

\newcommand{\forestTC}{
\tikz[planar forest ] {

\draw (0.0,1.0) arc (90:270:0.5);
\draw (0.0,1.0) arc (90:-90:0.5);
\node [draw=none] at (0.0, 0.0) { } 
child {node [b] at (0.0, 1.0) {  } edge from parent[draw=none] 
}
;
}}

\newcommand{\forestUC}{
\tikz[planar forest ] {

\draw (-2.220446049250313e-16,1.0) arc (90:270:0.5);
\draw (-2.220446049250313e-16,1.0) arc (90:-90:0.5);
\node [draw=none] at (-2.220446049250313e-16, 0.0) { } 
child {node [b] at (0.0, 1.0) {  } edge from parent[draw=none] 
}
;
\draw (1.2,1.0) arc (90:270:0.5);
\draw (1.2,1.0) arc (90:-90:0.5);
\node [draw=none] at (1.2, 0.0) { } 
child {node [b] at (0.0, 1.0) {  } edge from parent[draw=none] 
child {node [b] at (0.0, 1.0) {  }  
}
}
;
}}

\newcommand{\forestVC}{
\tikz[planar forest ] {

\node [b] at (0.0, 0.0) {  } 
;
\node [b] at (1.0, 0.0) {  } 
child {node [b] at (0.0, 1.0) {  }  
}
;
}}

\newcommand{\forestWC}{
\tikz[planar forest ] {

\node [b] at (0.0, 0.0) {  } 
;
}}

\newcommand{\forestXC}{
\tikz[planar forest ] {

\node [b] at (0.0, 0.0) {  } 
child {node [b] at (0.0, 1.0) {  }  
}
;
}}

\newcommand{\forestYC}{
\tikz[planar forest ] {

\node [b] at (0.0, 0.0) {  } 
child {node [b] at (0.0, 1.0) {  }  
}
;
}}

\newcommand{\forestAD}{
\tikz[planar forest ] {

\node [b] at (0.0, 0.0) {  } 
;
}}

\newcommand{\forestBD}{
\tikz[planar forest ] {

\node [b] at (0.0, 0.0) {  } 
;
\node [b] at (1.0, 0.0) {  } 
child {node [b] at (0.0, 1.0) {  }  
}
;
}}

\newcommand{\forestCD}{
\tikz[planar forest ] {

\draw (0.0,1.0) arc (90:270:0.5);
\draw (0.0,1.0) arc (90:-90:0.5);
\node [draw=none] at (0.0, 0.0) { } 
child {node [b] at (0.0, 1.0) {  } edge from parent[draw=none] 
}
;
\node [b] at (1.1, 0.0) {  } 
;
}}

\newcommand{\forestDD}{
\tikz[planar forest ] {

\draw (0.0,1.0) arc (90:270:0.5);
\draw (0.0,1.0) arc (90:-90:0.5);
\node [draw=none] at (0.0, 0.0) { } 
child {node [b] at (0.0, 1.0) {  } edge from parent[draw=none] 
child {node [b] at (0.0, 1.0) {  }  
}
}
;
\node [b] at (1.1, 0.0) {  } 
child {node [b] at (0.0, 1.0) {  }  
}
;
}}

\newcommand{\forestED}{
\tikz[planar forest ] {

\draw (0.0,1.0) arc (90:270:0.5);
\draw (0.0,1.0) arc (90:-90:0.5);
\node [draw=none] at (0.0, 0.0) { } 
child {node [b] at (0.0, 1.0) {  } edge from parent[draw=none] 
}
;
\node [b] at (1.1, 0.0) {  } 
;
}}

\newcommand{\forestFD}{
\tikz[planar forest ] {

\draw (0.0,1.0) arc (90:270:0.5);
\draw (0.0,1.0) arc (90:-90:0.5);
\node [draw=none] at (0.0, 0.0) { } 
child {node [b] at (0.0, 1.0) {  } edge from parent[draw=none] 
child {node [b] at (0.0, 1.0) {  }  
}
}
;
\node [b] at (1.1, 0.0) {  } 
child {node [b] at (0.0, 1.0) {  }  
}
;
}}

\newcommand{\forestGD}{
\tikz[planar forest ] {

\draw (0.0,1.0) arc (90:270:0.5);
\draw (0.0,1.0) arc (90:-90:0.5);
\node [draw=none] at (0.0, 0.0) { } 
child {node [b] at (0.0, 1.0) {  } edge from parent[draw=none] 
}
;
\node [b] at (1.1, 0.0) {  } 
;
}}

\newcommand{\forestHD}{
\tikz[planar forest ] {

\draw (0.0,1.0) arc (90:270:0.5);
\draw (0.0,1.0) arc (90:-90:0.5);
\node [draw=none] at (0.0, 0.0) { } 
child {node [b] at (0.0, 1.0) {  } edge from parent[draw=none] 
child {node [b] at (0.0, 1.0) {  }  
}
}
;
\node [b] at (1.1, 0.0) {  } 
child {node [b] at (0.0, 1.0) {  }  
}
;
}}

\newcommand{\forestID}{
\tikz[planar forest ] {

\draw (0.0,1.0) arc (90:270:0.5);
\draw (0.0,1.0) arc (90:-90:0.5);
\node [draw=none] at (0.0, 0.0) { } 
child {node [b] at (0.0, 1.0) {  } edge from parent[draw=none] 
child {node [b] at (0.0, 1.0) {  }  
}
}
;
\node [b] at (1.1, 0.0) {  } 
child {node [b] at (0.0, 1.0) {  }  
}
;
}}

\newcommand{\forestJD}{
\tikz[planar forest ] {

\draw (0.0,1.0) arc (90:270:0.5);
\draw (0.0,1.0) arc (90:-90:0.5);
\node [draw=none] at (0.0, 0.0) { } 
child {node [b] at (0.0, 1.0) {  } edge from parent[draw=none] 
}
;
\node [b] at (1.1, 0.0) {  } 
;
}}

\newcommand{\forestKD}{
\tikz[planar forest ] {

\draw (0.0,1.0) arc (90:270:0.5);
\draw (0.0,1.0) arc (90:-90:0.5);
\node [draw=none] at (0.0, 0.0) { } 
child {node [b] at (0.0, 1.0) {  } edge from parent[draw=none] 
}
;
\node [b] at (1.1, 0.0) {  } 
;
}}

\newcommand{\forestLD}{
\tikz[planar forest ] {

\draw (0.0,1.0) arc (90:270:0.5);
\draw (0.0,1.0) arc (90:-90:0.5);
\node [draw=none] at (0.0, 0.0) { } 
child {node [b] at (0.0, 1.0) {  } edge from parent[draw=none] 
child {node [b] at (0.0, 1.0) {  }  
}
}
;
\node [b] at (1.1, 0.0) {  } 
child {node [b] at (0.0, 1.0) {  }  
}
;
}}

\newcommand{\forestMD}{
\tikz[planar forest ] {

\draw (-4.440892098500626e-16,1.0) arc (90:270:0.5);
\draw (-4.440892098500626e-16,1.0) arc (90:-90:0.5);
\node [draw=none] at (-4.440892098500626e-16, 0.0) { } 
child {node [b] at (0.0, 1.0) {  } edge from parent[draw=none] 
}
;
\draw (1.2,1.0) arc (90:270:0.5);
\draw (1.2,1.0) arc (90:-90:0.5);
\node [draw=none] at (1.2, 0.0) { } 
child {node [b] at (0.0, 1.0) {  } edge from parent[draw=none] 
child {node [b] at (0.0, 1.0) {  }  
}
}
;
\node [b] at (2.3, 0.0) {  } 
;
}}

\newcommand{\forestND}{
\tikz[planar forest ] {

\draw (-4.440892098500626e-16,1.0) arc (90:270:0.5);
\draw (-4.440892098500626e-16,1.0) arc (90:-90:0.5);
\node [draw=none] at (-4.440892098500626e-16, 0.0) { } 
child {node [b] at (0.0, 1.0) {  } edge from parent[draw=none] 
}
;
\draw (1.2,1.0) arc (90:270:0.5);
\draw (1.2,1.0) arc (90:-90:0.5);
\node [draw=none] at (1.2, 0.0) { } 
child {node [b] at (0.0, 1.0) {  } edge from parent[draw=none] 
child {node [b] at (0.0, 1.0) {  }  
}
}
;
\node [b] at (2.3, 0.0) {  } 
;
}}

\newcommand{\forestOD}{
\tikz[planar forest ] {

\draw (0.0,1.0) arc (90:270:0.5);
\draw (0.0,1.0) arc (90:-90:0.5);
\node [draw=none] at (0.0, 0.0) { } 
child {node [b] at (0.0, 1.0) {  } edge from parent[draw=none] 
}
;
}}

\newcommand{\forestPD}{
\tikz[planar forest ] {

\draw (0.0,1.0) arc (90:270:0.5);
\draw (0.0,1.0) arc (90:-90:0.5);
\node [draw=none] at (0.0, 0.0) { } 
child {node [b] at (0.0, 1.0) {  } edge from parent[draw=none] 
}
;
}}

\newcommand{\forestQD}{
\tikz[planar forest ] {

\draw (0.0,1.0) arc (90:270:0.5);
\draw (0.0,1.0) arc (90:-90:0.5);
\node [draw=none] at (0.0, 0.0) { } 
child {node [b] at (0.0, 1.0) {  } edge from parent[draw=none] 
child {node [b] at (0.0, 1.0) {  }  
}
}
;
}}

\newcommand{\forestRD}{
\tikz[planar forest ] {

\node [b] at (0.0, 0.0) {  } 
;
}}

\newcommand{\forestSD}{
\tikz[planar forest ] {

\draw (0.0,1.0) arc (90:270:0.5);
\draw (0.0,1.0) arc (90:-90:0.5);
\node [draw=none] at (0.0, 0.0) { } 
child {node [b] at (0.0, 1.0) {  } edge from parent[draw=none] 
}
;
}}

\newcommand{\forestTD}{
\tikz[planar forest ] {

\draw (0.0,1.0) arc (90:270:0.5);
\draw (0.0,1.0) arc (90:-90:0.5);
\node [draw=none] at (0.0, 0.0) { } 
child {node [b] at (0.0, 1.0) {  } edge from parent[draw=none] 
child {node [b] at (0.0, 1.0) {  }  
}
}
;
}}

\newcommand{\forestUD}{
\tikz[planar forest ] {

\node [b] at (0.0, 0.0) {  } 
;
}}

\newcommand{\forestVD}{
\tikz[planar forest ] {

\node [b] at (0.0, 0.0) {  } 
;
}}

\newcommand{\forestWD}{
\tikz[planar forest ] {

\draw (-4.440892098500626e-16,1.0) arc (90:270:0.5);
\draw (-4.440892098500626e-16,1.0) arc (90:-90:0.5);
\node [draw=none] at (-4.440892098500626e-16, 0.0) { } 
child {node [b] at (0.0, 1.0) {  } edge from parent[draw=none] 
}
;
\draw (1.2,1.0) arc (90:270:0.5);
\draw (1.2,1.0) arc (90:-90:0.5);
\node [draw=none] at (1.2, 0.0) { } 
child {node [b] at (0.0, 1.0) {  } edge from parent[draw=none] 
child {node [b] at (0.0, 1.0) {  }  
}
}
;
\node [b] at (2.3, 0.0) {  } 
;
}}

\newcommand{\forestXD}{
\tikz[planar forest ] {

\draw (-4.440892098500626e-16,1.0) arc (90:270:0.5);
\draw (-4.440892098500626e-16,1.0) arc (90:-90:0.5);
\node [draw=none] at (-4.440892098500626e-16, 0.0) { } 
child {node [b] at (0.0, 1.0) {  } edge from parent[draw=none] 
}
;
\draw (1.2,1.0) arc (90:270:0.5);
\draw (1.2,1.0) arc (90:-90:0.5);
\node [draw=none] at (1.2, 0.0) { } 
child {node [b] at (0.0, 1.0) {  } edge from parent[draw=none] 
child {node [b] at (0.0, 1.0) {  }  
}
}
;
\node [b] at (2.3, 0.0) {  } 
;
}}

\newcommand{\forestYD}{
\tikz[planar forest ] {

\draw (0.0,1.0) arc (90:270:0.5);
\draw (0.0,1.0) arc (90:-90:0.5);
\node [draw=none] at (0.0, 0.0) { } 
child {node [b] at (0.0, 1.0) {  } edge from parent[draw=none] 
}
;
\node [b] at (1.1, 0.0) {  } 
;
}}

\newcommand{\forestAE}{
\tikz[planar forest ] {

\draw (0.0,1.0) arc (90:270:0.5);
\draw (0.0,1.0) arc (90:-90:0.5);
\node [draw=none] at (0.0, 0.0) { } 
child {node [b] at (0.0, 1.0) {  } edge from parent[draw=none] 
}
;
\node [b] at (1.1, 0.0) {  } 
;
}}

\newcommand{\forestBE}{
\tikz[planar forest ] {

\node [b] at (0.0, 0.0) {  } 
;
}}

\newcommand{\forestCE}{
\tikz[planar forest ] {

\draw (-4.440892098500626e-16,1.0) arc (90:270:0.5);
\draw (-4.440892098500626e-16,1.0) arc (90:-90:0.5);
\node [draw=none] at (-4.440892098500626e-16, 0.0) { } 
child {node [b] at (0.0, 1.0) {  } edge from parent[draw=none] 
}
;
\draw (1.2,1.0) arc (90:270:0.5);
\draw (1.2,1.0) arc (90:-90:0.5);
\node [draw=none] at (1.2, 0.0) { } 
child {node [b] at (0.0, 1.0) {  } edge from parent[draw=none] 
}
;
\node [b] at (2.3, 0.0) {  } 
;
}}

\newcommand{\forestDE}{
\tikz[planar forest ] {

\draw (-4.440892098500626e-16,1.0) arc (90:270:0.5);
\draw (-4.440892098500626e-16,1.0) arc (90:-90:0.5);
\node [draw=none] at (-4.440892098500626e-16, 0.0) { } 
child {node [b] at (0.0, 1.0) {  } edge from parent[draw=none] 
}
;
\draw (1.2,1.0) arc (90:270:0.5);
\draw (1.2,1.0) arc (90:-90:0.5);
\node [draw=none] at (1.2, 0.0) { } 
child {node [b] at (0.0, 1.0) {  } edge from parent[draw=none] 
child {node [b] at (0.0, 1.0) {  }  
}
}
;
\node [b] at (2.3, 0.0) {  } 
;
}}

\newcommand{\forestEE}{
\tikz[planar forest ] {

\draw (-1.0,1.0) edge[-] (0.0,1.0);
\draw (0.0,1.0) edge[-] (1.0,1.0);
\draw (-1.0,1.0) arc (90:270:0.5);
\draw (1.0,1.0) arc (90:-90:0.5);
\draw (-1.0,0.0) edge[-] (1.0,0.0);
\node [draw=none] at (0.0, 0.0) { } 
child {node [b] at (-1.0, 1.0) {  } edge from parent[draw=none] 
child {node [b] at (0.0, 1.0) {  }  
child {node [b] at (0.0, 1.0) {  }  
}
}
}
child {node [b] at (0.0, 1.0) {  } edge from parent[draw=none] 
}
child {node [b] at (1.0, 1.0) {  } edge from parent[draw=none] 
child {node [b] at (0.0, 1.0) {  }  
}
}
;
}}

\newcommand{\forestFE}{
\tikz[planar forest ] {

\node [b] at (0.0, 0.0) {  } 
child {node [b] at (0.0, 1.0) {  }  
}
;
}}

\newcommand{\forestGE}{
\tikz[planar forest ] {

\node [b] at (0.0, 0.0) {  } 
;
}}

\newcommand{\forestHE}{
\tikz[planar forest ] {

\node [b] at (0.0, 0.0) {  } 
;
}}

\newcommand{\forestIE}{
\tikz[planar forest ] {

\node [b] at (0.0, 0.0) {  } 
child {node [b] at (0.0, 1.0) {  }  
}
;
\node [b] at (1.0, 0.0) {  } 
;
}}

\newcommand{\forestJE}{
\tikz[planar forest ] {

\node [b] at (0.0, 0.0) {  } 
;
\node [b] at (1.0, 0.0) {  } 
;
}}

\newcommand{\forestKE}{
\tikz[planar forest ] {

\draw (-1.0,1.0) edge[-] (0.0,1.0);
\draw (0.0,1.0) edge[-] (1.0,1.0);
\draw (-1.0,1.0) arc (90:270:0.5);
\draw (1.0,1.0) arc (90:-90:0.5);
\draw (-1.0,0.0) edge[-] (1.0,0.0);
\node [draw=none] at (0.0, 0.0) { } 
child {node [b] at (-1.0, 1.0) {  } edge from parent[draw=none] 
child {node [b] at (0.0, 1.0) {  }  
child {node [b] at (0.0, 1.0) {  }  
}
}
}
child {node [b] at (0.0, 1.0) {  } edge from parent[draw=none] 
}
child {node [b] at (1.0, 1.0) {  } edge from parent[draw=none] 
child {node [b] at (0.0, 1.0) {  }  
}
}
;
}}

\newcommand{\forestLE}{
\tikz[planar forest ] {

\draw (-1.0,1.0) edge[-] (0.0,1.0);
\draw (0.0,1.0) edge[-] (1.0,1.0);
\draw (-1.0,1.0) arc (90:270:0.5);
\draw (1.0,1.0) arc (90:-90:0.5);
\draw (-1.0,0.0) edge[-] (1.0,0.0);
\node [draw=none] at (0.0, 0.0) { } 
child {node [b] at (-1.0, 1.0) {  } edge from parent[draw=none] 
}
child {node [b] at (0.0, 1.0) {  } edge from parent[draw=none] 
}
child {node [b] at (1.0, 1.0) {  } edge from parent[draw=none] 
child {node [b] at (0.0, 1.0) {  }  
}
}
;
}}

\newcommand{\forestME}{
\tikz[planar forest ] {

\draw (-1.0,1.0) edge[-] (0.0,1.0);
\draw (0.0,1.0) edge[-] (1.0,1.0);
\draw (-1.0,1.0) arc (90:270:0.5);
\draw (1.0,1.0) arc (90:-90:0.5);
\draw (-1.0,0.0) edge[-] (1.0,0.0);
\node [draw=none] at (0.0, 0.0) { } 
child {node [b] at (-1.0, 1.0) {  } edge from parent[draw=none] 
child {node [b] at (0.0, 1.0) {  }  
}
}
child {node [b] at (0.0, 1.0) {  } edge from parent[draw=none] 
}
child {node [b] at (1.0, 1.0) {  } edge from parent[draw=none] 
child {node [b] at (0.0, 1.0) {  }  
}
}
;
}}

\newcommand{\forestNE}{
\tikz[planar forest ] {

\draw (-1.0,1.0) edge[-] (0.0,1.0);
\draw (0.0,1.0) edge[-] (1.0,1.0);
\draw (-1.0,1.0) arc (90:270:0.5);
\draw (1.0,1.0) arc (90:-90:0.5);
\draw (-1.0,0.0) edge[-] (1.0,0.0);
\node [draw=none] at (0.0, 0.0) { } 
child {node [b] at (-1.0, 1.0) {  } edge from parent[draw=none] 
child {node [b] at (0.0, 1.0) {  }  
child {node [b] at (0.0, 1.0) {  }  
}
}
}
child {node [b] at (0.0, 1.0) {  } edge from parent[draw=none] 
}
child {node [b] at (1.0, 1.0) {  } edge from parent[draw=none] 
}
;
}}

\newcommand{\forestOE}{
\tikz[planar forest ] {

\draw (-1.0,1.0) edge[-] (0.0,1.0);
\draw (0.0,1.0) edge[-] (1.0,1.0);
\draw (-1.0,1.0) arc (90:270:0.5);
\draw (1.0,1.0) arc (90:-90:0.5);
\draw (-1.0,0.0) edge[-] (1.0,0.0);
\node [draw=none] at (0.0, 0.0) { } 
child {node [b] at (-1.0, 1.0) {  } edge from parent[draw=none] 
}
child {node [b] at (0.0, 1.0) {  } edge from parent[draw=none] 
}
child {node [b] at (1.0, 1.0) {  } edge from parent[draw=none] 
}
;
}}

\newcommand{\forestPE}{
\tikz[planar forest ] {

\draw (-1.0,1.0) edge[-] (0.0,1.0);
\draw (0.0,1.0) edge[-] (1.0,1.0);
\draw (-1.0,1.0) arc (90:270:0.5);
\draw (1.0,1.0) arc (90:-90:0.5);
\draw (-1.0,0.0) edge[-] (1.0,0.0);
\node [draw=none] at (0.0, 0.0) { } 
child {node [b] at (-1.0, 1.0) {  } edge from parent[draw=none] 
child {node [b] at (0.0, 1.0) {  }  
}
}
child {node [b] at (0.0, 1.0) {  } edge from parent[draw=none] 
}
child {node [b] at (1.0, 1.0) {  } edge from parent[draw=none] 
}
;
}}

\newcommand{\forestQE}{
\tikz[planar forest ] {

\draw (0.0,1.0) arc (90:270:0.5);
\draw (0.0,1.0) arc (90:-90:0.5);
\node [draw=none] at (0.0, 0.0) { } 
child {node [b] at (0.0, 1.0) {  } edge from parent[draw=none] 
child {node [b] at (0.0, 1.0) {  }  
}
}
;
\node [b] at (1.1, 0.0) {  } 
;
}}

\newcommand{\forestRE}{
\tikz[planar forest ] {

\draw (-0.5,1.0) edge[-] (0.5,1.0);
\draw (-0.5,1.0) arc (90:270:0.5);
\draw (0.5,1.0) arc (90:-90:0.5);
\draw (-0.5,0.0) edge[-] (0.5,0.0);
\node [draw=none] at (0.0, 0.0) { } 
child {node [b] at (-0.5, 1.0) {  } edge from parent[draw=none] 
}
child {node [b] at (0.5, 1.0) {  } edge from parent[draw=none] 
}
;
\node [b] at (1.6, 0.0) {  } 
;
}}

\newcommand{\forestSE}{
\tikz[planar forest ] {

\node [b] at (0.0, 0.0) {  } 
child {node [b] at (-0.5, 1.0) {  }  
}
child {node [b] at (0.5, 1.0) {  }  
}
;
}}

\newcommand{\forestTE}{
\tikz[planar forest ] {

\draw (0.0,1.0) arc (90:270:0.5);
\draw (0.0,1.0) arc (90:-90:0.5);
\node [draw=none] at (0.0, 0.0) { } 
child {node [b] at (0.0, 1.0) {  } edge from parent[draw=none] 
}
;
\node [b] at (1.1, 0.0) {  } 
child {node [b] at (0.0, 1.0) {  }  
}
;
}}

\newcommand{\forestUE}{
\tikz[planar forest ] {

\draw (0.0,1.0) arc (90:270:0.5);
\draw (0.0,1.0) arc (90:-90:0.5);
\node [draw=none] at (0.0, 0.0) { } 
child {node [b] at (0.0, 1.0) {  } edge from parent[draw=none] 
child {node [b] at (0.0, 1.0) {  }  
child {node [b] at (0.0, 1.0) {  }  
}
}
}
;
\node [b] at (1.1, 0.0) {  } 
;
}}

\newcommand{\forestVE}{
\tikz[planar forest ] {

\draw (-0.5,1.0) edge[-] (0.5,1.0);
\draw (-0.5,1.0) arc (90:270:0.5);
\draw (0.5,1.0) arc (90:-90:0.5);
\draw (-0.5,0.0) edge[-] (0.5,0.0);
\node [draw=none] at (0.0, 0.0) { } 
child {node [b] at (-0.5, 1.0) {  } edge from parent[draw=none] 
}
child {node [b] at (0.5, 1.0) {  } edge from parent[draw=none] 
child {node [b] at (0.0, 1.0) {  }  
}
}
;
\node [b] at (1.6, 0.0) {  } 
;
}}

\newcommand{\forestWE}{
\tikz[planar forest ] {

\draw (-1.0,1.0) edge[-] (0.0,1.0);
\draw (0.0,1.0) edge[-] (1.0,1.0);
\draw (-1.0,1.0) arc (90:270:0.5);
\draw (1.0,1.0) arc (90:-90:0.5);
\draw (-1.0,0.0) edge[-] (1.0,0.0);
\node [draw=none] at (0.0, 0.0) { } 
child {node [b] at (-1.0, 1.0) {  } edge from parent[draw=none] 
}
child {node [b] at (0.0, 1.0) {  } edge from parent[draw=none] 
}
child {node [b] at (1.0, 1.0) {  } edge from parent[draw=none] 
}
;
\node [b] at (2.1, 0.0) {  } 
;
}}

\newcommand{\forestXE}{
\tikz[planar forest ] {

\node [b] at (0.0, 0.0) {  } 
child {node [b] at (-0.5, 1.0) {  }  
}
child {node [b] at (0.5, 1.0) {  }  
child {node [b] at (0.0, 1.0) {  }  
}
}
;
}}

\newcommand{\forestYE}{
\tikz[planar forest ] {

\node [b] at (0.0, 0.0) {  } 
child {node [b] at (0.0, 1.0) {  }  
child {node [b] at (-0.5, 1.0) {  }  
}
child {node [b] at (0.5, 1.0) {  }  
}
}
;
}}

\newcommand{\forestAF}{
\tikz[planar forest ] {

\draw (0.0,1.0) arc (90:270:0.5);
\draw (0.0,1.0) arc (90:-90:0.5);
\node [draw=none] at (0.0, 0.0) { } 
child {node [b] at (0.0, 1.0) {  } edge from parent[draw=none] 
}
;
\node [b] at (1.1, 0.0) {  } 
child {node [b] at (0.0, 1.0) {  }  
child {node [b] at (0.0, 1.0) {  }  
}
}
;
}}

\newcommand{\forestBF}{
\tikz[planar forest ] {

\draw (0.0,1.0) arc (90:270:0.5);
\draw (0.0,1.0) arc (90:-90:0.5);
\node [draw=none] at (0.0, 0.0) { } 
child {node [b] at (0.0, 1.0) {  } edge from parent[draw=none] 
child {node [b] at (-0.5, 1.0) {  }  
}
child {node [b] at (0.5, 1.0) {  }  
}
}
;
\node [b] at (1.1, 0.0) {  } 
;
}}

\newcommand{\forestCF}{
\tikz[planar forest ] {

\draw (-0.5,1.0) edge[-] (0.5,1.0);
\draw (-0.5,1.0) arc (90:270:0.5);
\draw (0.5,1.0) arc (90:-90:0.5);
\draw (-0.5,0.0) edge[-] (0.5,0.0);
\node [draw=none] at (0.0, 0.0) { } 
child {node [b] at (-0.5, 1.0) {  } edge from parent[draw=none] 
}
child {node [b] at (0.5, 1.0) {  } edge from parent[draw=none] 
child {node [b] at (0.0, 1.0) {  }  
}
}
;
\node [b] at (1.6, 0.0) {  } 
;
}}

\newcommand{\forestDF}{
\tikz[planar forest ] {

\node [b] at (0.0, 0.0) {  } 
child {node [b] at (0.0, 1.0) {  }  
child {node [b] at (-0.5, 1.0) {  }  
}
child {node [b] at (0.5, 1.0) {  }  
}
}
;
}}

\newcommand{\forestEF}{
\tikz[planar forest ] {

\node [b] at (0.0, 0.0) {  } 
child {node [b] at (-0.5, 1.0) {  }  
}
child {node [b] at (0.5, 1.0) {  }  
child {node [b] at (0.0, 1.0) {  }  
}
}
;
}}

\newcommand{\forestFF}{
\tikz[planar forest ] {

\node [b] at (0.0, 0.0) {  } 
child {node [b] at (-1.0, 1.0) {  }  
}
child {node [b] at (0.0, 1.0) {  }  
}
child {node [b] at (1.0, 1.0) {  }  
}
;
}}

\newcommand{\forestGF}{
\tikz[planar forest ] {

\draw (0.0,1.0) arc (90:270:0.5);
\draw (0.0,1.0) arc (90:-90:0.5);
\node [draw=none] at (0.0, 0.0) { } 
child {node [b] at (0.0, 1.0) {  } edge from parent[draw=none] 
}
;
\node [b] at (1.5, 0.0) {  } 
child {node [b] at (-0.5, 1.0) {  }  
}
child {node [b] at (0.5, 1.0) {  }  
}
;
}}

\newcommand{\forestHF}{
\tikz[planar forest ] {

\draw (-4.440892098500626e-16,1.0) arc (90:270:0.5);
\draw (-4.440892098500626e-16,1.0) arc (90:-90:0.5);
\node [draw=none] at (-4.440892098500626e-16, 0.0) { } 
child {node [b] at (0.0, 1.0) {  } edge from parent[draw=none] 
}
;
\draw (1.2,1.0) arc (90:270:0.5);
\draw (1.2,1.0) arc (90:-90:0.5);
\node [draw=none] at (1.2, 0.0) { } 
child {node [b] at (0.0, 1.0) {  } edge from parent[draw=none] 
child {node [b] at (0.0, 1.0) {  }  
}
}
;
\node [b] at (2.3, 0.0) {  } 
;
}}

\newcommand{\forestIF}{
\tikz[planar forest ] {

\draw (4.440892098500626e-16,1.0) arc (90:270:0.5);
\draw (4.440892098500626e-16,1.0) arc (90:-90:0.5);
\node [draw=none] at (4.440892098500626e-16, 0.0) { } 
child {node [b] at (0.0, 1.0) {  } edge from parent[draw=none] 
}
;
\draw (1.2,1.0) edge[-] (2.2,1.0);
\draw (1.2,1.0) arc (90:270:0.5);
\draw (2.2,1.0) arc (90:-90:0.5);
\draw (1.2,0.0) edge[-] (2.2,0.0);
\node [draw=none] at (1.7, 0.0) { } 
child {node [b] at (-0.5, 1.0) {  } edge from parent[draw=none] 
}
child {node [b] at (0.5000000000000002, 1.0) {  } edge from parent[draw=none] 
}
;
\node [b] at (3.3, 0.0) {  } 
;
}}

\newcommand{\forestJF}{
\tikz[planar forest ] {

\draw (0.0,1.0) arc (90:270:0.5);
\draw (0.0,1.0) arc (90:-90:0.5);
\node [draw=none] at (0.0, 0.0) { } 
child {node [b] at (0.0, 1.0) {  } edge from parent[draw=none] 
child {node [b] at (0.0, 1.0) {  }  
child {node [b] at (0.0, 1.0) {  }  
}
}
}
;
\node [b] at (1.1, 0.0) {  } 
;
}}

\newcommand{\forestKF}{
\tikz[planar forest ] {

\draw (0.0,1.0) arc (90:270:0.5);
\draw (0.0,1.0) arc (90:-90:0.5);
\node [draw=none] at (0.0, 0.0) { } 
child {node [b] at (0.0, 1.0) {  } edge from parent[draw=none] 
}
;
\node [b] at (1.5, 0.0) {  } 
child {node [b] at (-0.5, 1.0) {  }  
}
child {node [b] at (0.5, 1.0) {  }  
}
;
}}

\newcommand{\forestLF}{
\tikz[planar forest ] {

\draw (-4.440892098500626e-16,1.0) arc (90:270:0.5);
\draw (-4.440892098500626e-16,1.0) arc (90:-90:0.5);
\node [draw=none] at (-4.440892098500626e-16, 0.0) { } 
child {node [b] at (0.0, 1.0) {  } edge from parent[draw=none] 
}
;
\draw (1.2,1.0) arc (90:270:0.5);
\draw (1.2,1.0) arc (90:-90:0.5);
\node [draw=none] at (1.2, 0.0) { } 
child {node [b] at (0.0, 1.0) {  } edge from parent[draw=none] 
}
;
\node [b] at (2.3, 0.0) {  } 
child {node [b] at (0.0, 1.0) {  }  
}
;
}}

\newcommand{\forestMF}{
\tikz[planar forest ] {

\draw (0.0,1.0) arc (90:270:0.5);
\draw (0.0,1.0) arc (90:-90:0.5);
\node [draw=none] at (0.0, 0.0) { } 
child {node [b] at (0.0, 1.0) {  } edge from parent[draw=none] 
child {node [b] at (0.0, 1.0) {  }  
}
}
;
\node [b] at (1.1, 0.0) {  } 
child {node [b] at (0.0, 1.0) {  }  
}
;
}}

\usepackage{graphicx}
\usepackage[all]{xy}
\usepackage{axodraw}
\usepackage{shuffle}

\newcolumntype{C}{>{$}c<{$}}

\allowdisplaybreaks

\newtheorem{theorem}{Theorem}[section]
\newtheorem{definition}[theorem]{Definition}
\newtheorem*{definition*}{Definition}

\newtheorem{proposition}[theorem]{Proposition}
\newtheorem{lemma}[theorem]{Lemma}
\newtheorem{remark}[theorem]{Remark}
\newtheorem*{remark*}{Remark}

\newtheorem*{remarks*}{Remarks}
\newtheorem{corollary}[theorem]{Corollary}

\newtheorem*{notation*}{Notation}
\newtheorem{ex}[theorem]{Example}
\newtheorem*{ex*}{Example}

\newtheorem*{exs*}{Examples}

\newtheorem*{app*}{Application}
\newtheorem{conjecture*}{Conjecture}
\newcommand{\kk}{\mathbf k}

\newcommand{\pil}{\rightarrow}

\newcommand{\End}{\text{End}}

\newcommand{\Der}{\text{Der}}

\newcommand{\ignore}[1]{{}}

\def\fleche#1{\mathop{\hbox to #1 mm{\rightarrowfill}}\limits}
\def\gfleche#1{\mathop{\hbox to #1 mm{\leftarrowfill}}\limits}
\def\inj#1{\mathop{\hbox to #1 mm{$\lhook\joinrel$\rightarrowfill}}\limits}
\def\ginj#1{\mathop{\hbox to #1 mm{\leftarrowfill$\joinrel\rhook$}}\limits}
\def\surj#1{\mathop{\hbox to #1 mm{\rightarrowfill\hskip 2pt\llap{$\rightarrow$}}}\limits}
\def\gsurj#1{\mathop{\hbox to #1 mm{\rlap{$\leftarrow$}\hskip 2pt \leftarrowfill}}\limits}

 \newcommand{\delete}[1]{}

\newcommand{\mop}[1]{\mathop{\hbox {\rm #1}}\nolimits}

\def \restr#1{\mathstrut_{\textstyle |}\raise-8pt\hbox{$\scriptstyle #1$}}
\def \srestr#1{\mathstrut_{\scriptstyle |}\hbox to
  -1.5pt{}\raise-4pt\hbox{$\scriptscriptstyle #1$}}

\def\shift#1{\mathop{#1}^{\leftarrow}\limits}

\def\semi{\mathrel{\times}\kern -6.5pt\joinrel\mathrel{\raise 1.4pt\hbox{${\scriptscriptstyle |}$}}\kern 2pt}
\title{
Aromatic and clumped multi-indices: algebraic structure and Hopf embeddings
}

\author{
Zhicheng Zhu\textsuperscript{1} and Adrien Busnot Laurent\textsuperscript{2}
}

\begin{document}

\footnotetext[1]{
School of Mathematics and Statistics, Lanzhou University Lanzhou, 730000, China. zhuzhch16@lzu.edu.cn.}
\footnotetext[2]{
Univ Rennes, INRIA (Research team MINGuS), IRMAR (CNRS UMR 6625) and ENS Rennes, France. Adrien.Busnot-Laurent@inria.fr.}

\maketitle

\begin{abstract}

Butcher forests extend naturally into aromatic and clumped forests and play a fundamental role in the numerical analysis of volume-preserving methods. The design of general volume-preserving methods is a challenging open problem, and recent attempts showed progress on specific dynamics.
We introduce aromatic and clumped multi-indices, that are algebraic objects that simplify the study of volume-preservation to the one-dimensional setting, while retaining much of the structure (in stark opposition to standard multi-indices). We provide their algebraic structure of pre-Lie-Rinehart algebra, Hopf algebroid, and Hopf algebra, apply them in numerical analysis, and generalise to the aromatic context the Hopf embedding from multi-indices to the BCK Hopf algebra.

\smallskip
\noindent
{\it Keywords:\,} multi-indices, aromatic trees, Novikov algebra, pre-Lie-Rinehart algebra, Hopf algebroid, Hopf algebra, geometric numerical integration, Butcher series.
\smallskip

\noindent
{\it AMS subject classification (2020):\,}  16T05, 37M15, 05C05, 16T30, 17A30.
 \end{abstract}

%\tableofcontents

\section{Introduction}
\label{section:Introduction}

While Butcher trees were introduced for the high order analysis of Runge-Kutta methods \cite{Butcher72aat, Hairer06gni}, aromatic trees are an extension that includes the use of specific graphs, in order to represent terms such as the divergence operator. They play a crucial role in the study of volume-preserving integrators \cite{Chartier07pfi, Iserles07bsm, Bogfjellmo22uat} and are also studied for their fundamental algebraic and geometric properties \cite{MuntheKaas16abs, McLachlan16bsm, Bogfjellmo19aso, Laurent23tab, Laurent23tld}, as well as for other applications, for instance in stochastic numerics \cite{Bronasco22cef}.
They were also recently extended for the study of volume-preserving methods on manifolds \cite{Busnot26tft} and yield important examples of post-Lie-Rinehart and post-Hopf algebroid structures \cite{Busnot25pha}.
In particular, the standard Butcher forests and the aromatic forests have Hopf algebra structures \cite{Connes98har,Chartier10aso,Bogfjellmo19aso,Bronasco22cef}, which play a role in the creation of numerical methods, but were also used, for instance in renormalisation of QFT, or for the study of singular SPDEs. We are especially interested in the Butcher-Connes-Kreimer Hopf algebra of forests $\HH_{BCK}$ and aromatic forests $\HH_{BCK}^{aro}$. As the algebraic structure proves challenging, the algebra of clumped forests $\HH_{BCK}^{cl}$ was later introduced in \cite{Bronasco22cef}.
It was observed in \cite{Bogfjellmo22uat} that for specific vector fields, several degeneracies occur with aromatic trees and simpler formalisms would help understand the full picture.
In numerics, multi-indices correspond to the case of dimension one, while Butcher trees correspond to the infinite dimensional case. Unfortunately, the problem of volume preservation becomes trivial when rewritten with mult-indices.
We present in this paper a new non-trivial algebraic object based on multi-indices allowing to represent Taylor expansions for the study of volume-preservation. This is a first step toward the characterisation of volume-preserving methods in any dimension.
We also mention that the question of exact volume-preservation for numerical integrators recently regained interest through its generalisation to the stochastic context. In this setting, it is shown in \cite{Laurent21ata, Bronasco22cef} that the design of methods for sampling exactly the invariant distribution of ergodic stochastic dynamics requires a deep understanding of so-called exotic aromatic B-series \cite{Laurent23tue}. The use of multi-indices structures is a first step toward the design of such discretisations, with numerous application in stochastic optimisation, molecular dynamics, and machine learning.

The Hopf algebra of multi-indices $\HH_{LOT}$ is introduced in the work \cite{Linares21tsg} for the study of rough PDEs with regularity structures.
The algebra was then further studied and extended in a variety of works including \cite{Linares22atf, Bruned23naa, Jacques23pla, Linares23ipl, Bruned24atd, Bruned25mib, Bruned25mic}.
From the numerical analysis viewpoint, multi-indices exactly correspond to the one-dimensional case: they identify the degeneracies of Butcher trees for ODEs in dimension one. Multi-indices thus serve as a toy model for tackling challenging numerical problems, before working on trees. We mention that there are no natural combinatorial structures between multi-indices and trees for representing ODEs in fixed dimension $d>1$, as pointed out in \cite{Bruned25edf}.
It is shown in \cite{Zhu24fna} that there exists a natural injective embedding $j$ that sends the Hopf algebra of multi-indices $\HH_{LOT}$ to the Hopf algebra of trees $\HH_{BCK}$.
The present paper introduces new spaces of multi-indices analogous to aromatic and clumped forests, mimicking the concepts of numerical analysis.
We present their Hopf algebra and pre-Lie-Rinehart algebra structures. Moreover, we show that there are Hopf embeddings $j^{aro}$ and $j^{cl}$ between the new multi-indices spaces and the aromatic and clumped forests, of which we give an overview in the following diagram.
\[
\begin{tikzcd}
\HH_{LOT}^{aro} \arrow[r, "j^{aro}",hook] \arrow[dd,"\varphi^*"' , bend right=49]& \HH_{BCK}^{aro} \arrow[dd,"\psi^*" , bend left=49]\\
\HH_{LOT} \arrow[r, "j",hook] \arrow[d, hook] \arrow[u, hook]& \HH_{BCK}\arrow[d, hook] \arrow[u, hook]\\
\HH_{LOT}^{cl} \arrow[r, "j^{cl}",hook]  & \HH_{BCK}^{cl}
\end{tikzcd}
\]

The paper is organised in the following way.
Section \ref{section:AF} presents a concise review of the algebraic structures based on aromatic trees in numerical analysis.
Section \ref{section:aromatic_multi-indices} presents the main results of this paper: we introduce the new aromatic and clumped multi-indices, provide their algebraic structure, show the Hopf embeddings between multi-indices and Butcher forests, and show how aromatic multi-indices apply in geometric numerical integration.
The proofs are presented in Section \ref{section:Hopf_structure} and rely on an aromatic extension of the Novikov algebra.

\section{Preliminaries on aromatic and clumped Butcher forests}
\label{section:AF}

This section is devoted to a concise overview of aromatic and clumped forests and their algebraic structures, as uncovered in \cite{Floystad20tup, Bogfjellmo19aso, Bronasco22cef}.
We shall introduce the aromatic extensions of multi-indices by mimicking the aromatic and clumped forests in Section \ref{section:aromatic_multi-indices}.

\subsection{The pre-Lie-Rinehart algebra of aromatic trees}

Let $C$ be a finite set, whose elements are called decorations (or colours in numerical analysis). We focus on trees with only one decoration for the examples for the sake of simplicity, but also as this is the case of major interest in numerics.
\begin{definition}
An aromatic tree is a directed graph $(V,E)$ with vertices $V$ and edges $E\subset V\times V$, where each vertex has exactly one outgoing edge, except one vertex called the root, that has none. The connected component with the root is called a tree, and the other connected components are called aromas. The empty aroma is denoted as $\mathbf{1}$. A multiset of aromas is called a multiaroma.
The vertices are decorated by elements of $C$, that is, we attach to each aromatic tree a map $d\colon V\rightarrow C$, often omitted in the notation for the sake of clarity.
We write $T$ the vector space of trees, $A_0$ the space of aromas, $\AA=S(A_0)$ the space of multiaromas, understood as the symmetric algebra over aromas, and $\AA T=\AA\otimes T$ the vector space of aromatic trees.
These spaces are naturally graded by the number of vertices, called the order of an aromatic tree $\pi$ and denoted $\abs{\pi}$.
\end{definition}

We draw aromatic trees as follows.
Aromas are drawn in an arbitrary order on the left of the tree. By definition, each aroma has exactly one cycle, also called $K$-loop in \cite{Iserles07bsm}, that is, a finite set of vertices $v_1,\dots,v_K$ such that there is an edge going from $v_1$ to $v_2$,\dots, $v_K$ to $v_1$. Thus, any aroma can be written as a tuple of trees $(t_1,\dots,t_n)$ with cyclic invariance.
The orientation of the edges goes from top to bottom and in clockwise order for cycles.
We find in particular
\begin{equation}
\label{ex:aroma_cyclic_inv}
\forestA=\forestB=\forestC.
\end{equation}
An aroma can also be described as a tuple of trees with cyclic invariance, obtained by removing the edges in the cycle of the aroma (see \cite{Busnot26tft}). For instance, the aroma in \eqref{ex:aroma_cyclic_inv} is equivalently represented by
\[
(\forestD,\forestE,\forestF)_\circlearrowleft=(\forestG,\forestH,\forestI)_\circlearrowleft=(\forestJ,\forestK,\forestL)_\circlearrowleft.
\]
The elements of order up to four of $\AA T$ (with one decoration only) are
\[
\forestM,\quad
\forestN,\quad \forestO,\quad
\forestP,\quad \forestQ,\quad \forestR,\quad \forestS,\quad \forestT,\quad \forestU,\quad
\forestV,\quad \forestW,\quad \forestX,\quad \forestY,\quad \forestAB,\]
\[
\forestBB,\quad \forestCB,\quad \forestDB,\quad \forestEB,\quad \forestFB,\quad \forestGB,\quad \forestHB,\quad \forestIB,\quad \forestJB,\quad \forestKB,\quad \forestLB.
\]
For clarity, we use the notation in bold $\mathbf{a}$ for describing multiaromas and $a$ for aromas.
Detailed combinatorics with aromatic trees are presented in \cite{Laurent23tab}.

We added aromas to the standard pre-Lie structure of Butcher trees and can now wonder which structure it yields on $\AA T$.
\begin{definition}
\label{def:div}
The grafting product $\curvearrowright\colon T\times \AA\rightarrow \AA$ and $\curvearrowright\colon T\times T\rightarrow T$ are defined by grafting the root of the first input to the nodes of the second input, summing over all vertices.
Then, the grafting product extends to aromatic trees $\curvearrowright\colon \AA T\times \AA T\rightarrow \AA T$ by
\[
(a_1\tau_1)\curvearrowright (a_2 \tau_2)=a_1 a_2 (\tau_1\curvearrowright \tau_2)+a_1  (\tau_1\curvearrowright a_2) \tau_2.\]
The divergence $d\colon \AA T\rightarrow \AA$ adds an edge going from the root to a node, summing over all vertices.
\end{definition}

For instance, we find
\[
\forestMB\curvearrowright\forestNB
=\forestOB+\forestPB+\forestQB+\forestRB,\quad
d \left(\forestSB\right)=\forestTB+\forestUB+2\forestVB.
\]

The aromatic trees naturally have a structure of pre-Lie-Rinehart algebra.
\begin{definition} \label{L-lieRinehart}
Let $\K$ be a field, $R$ be a unital commutative $\K$-algebra, and $L$ be a $R$-module, equipped with a $R$-linear map $\rho$, called the anchor map, and a map $\nabla$:
\[\rho: L \rightarrow Der_{\K}(R,R),\quad \nabla: L\rightarrow \End_{\kk}(L,L).\]
Then, $L$ is a pre-Lie-Rinehart algebra if, with the notation $X\rhd Y:=\nabla_{X}Y$ and $[X,Y]=X\rhd Y - Y\rhd X,$:
\begin{itemize}
\item $(L,[-,-])$ is a Lie algebra over $\K$,
\item The anchor map $\rho$ is a homomorphism of Lie algebras,
\item The Leibniz rule holds: for $f\in R$ and $X, Y\in L$, $[X, fY]=\big(\rho(X)fY\big) + f[X, Y]$,
\item $(L,\triangleright)$ is a pre-Lie algebra: for $X, Y, Z\in L$, one has
$$X \rhd (Y \rhd Z) - (X\rhd Y)\rhd Z = Y \rhd (X \rhd Z) - (Y \rhd X )\rhd Z.$$
\end{itemize}
\end{definition}

\begin{proposition}[\cite{Floystad20tup}]
The $\AA$-module $\AA T$ is a pre-Lie-Rinehart algebra when equipped with the maps
\[\rho(\tau)(\mathbf{a})=\tau\curvearrowright \mathbf{a},\quad \nabla_{\tau_1} \tau_2 =\tau_1\curvearrowright \tau_2.\]
\end{proposition}

Given a vertex $v$ of an aromatic tree $\tau\in \AA T$, let $\delta_v\in \End (\AA T)$ be the map that grafts the root of an aromatic tree on $v$. For any choice of a vertex $v$ in a rooted tree $t$, we represent $\delta_v(t)$ by attaching a free edge (indicated by a leaf $\forestWB$) to $v$. The space of aromatic trees with a free edge is denoted $\AA T^\times$.
Then, the trace $t\colon \AA T^\times\rightarrow \AA$ links the root to the free edge.
In particular, we find
\[
t(\forestXB)=\forestYB.
\]
The divergence then satisfies $d=t\circ \delta$, where $\delta=\sum_v\delta_v$.
Equipped with the trace, aromatic trees satisfy a universal algebraic property: the aromatic trees $(\AA T,\curvearrowright,t)$ over the $\R$-algebra $\AA$ are the free tracial pre-Lie-Rinehart algebra (see \cite{Floystad20tup} for details).
We mention that aromatic structures have been extended to the post-Lie context in \cite{Busnot26tft, Rahm26tup} and yield the free tracial post-Lie-Rinehart algebra.

\subsection{Hopf structure of aromatic and clumped forests}
\label{clumphopf}

In numerical analysis, Butcher trees represent vector fields and symmetric concatenation of trees, called forests, represent differential operators. The different Hopf algebra structures on Butcher forests yield important numerical results. For the analysis of splitting methods, the simpler Hopf algebras of words is used \cite{Blanes24smf}, while for Runge-Kutta methods, Hopf algebras of forests are the natural approach \cite{Chartier10aso, McLachlan16bsm}.
There are two main possibilities to build forests from aromatic trees: aromatic forests are the ones that appear in the Taylor expansions of numerical integrators, and clumped forests have a simpler algebraic structure that is useful for the formulation of backward error analysis. We refer to \cite{Bogfjellmo19aso, Bronasco22cef} for the details.
\begin{definition}
The $\AA$-module $\AA\FF=\AA\otimes S(T)$ is the space of aromatic forests and $\CC\FF=S(\AA T)$ is the (symmetric) algebra of clumped forests. They are graded by the number of nodes. The empty multiaroma and the empty forest are written with the same notation $\mathbf{1}$.
\end{definition}
We use the same notation $\cdot$ for all products between aromas and trees and omit it when the context is clear. We emphasize that this product is commutative.
We will use parenthesis to represent clumped forests:
\[
\forestAC\in \AA\FF, \quad
(\forestBC)\ \forestCC\neq (\forestDC)\ \forestEC\ (\forestFC)\in \CC\FF.
\]

Let the symmetry coefficient of an aromatic forest $\pi\in \AA\FF$ be the size of the automorphism group of the associated graph, where we recall that the automorphism group of a directed graph $(V,E)$ is the set of permutations $g\colon V\rightarrow V$ satisfying $(g\times g)(E)=E$. The symmetry coefficient on $\CC\FF$ is induced from the definition on $\AA T$ by
\[\sigma((\mathbf{a}^1\tau^1)^{r_1}\cdots(\mathbf{a}^n\tau^n)^{r_n})=\sigma^{ext}((\mathbf{a}^1\tau^1)^{r_1}\cdots(\mathbf{a}^n\tau^n)^{r_n})\sigma^{int}((\mathbf{a}^1\tau^1)^{r_1}\cdots(\mathbf{a}^n\tau^n)^{r_n}),\]
with the external and internal symmetry factors:
\begin{align*}
    \sigma^{ext}((\mathbf{a}^1\tau^1)^{r_1}\cdots(\mathbf{a}^n\tau^n)^{r_n})&=r_1!\cdots r_n!,\\
\sigma^{int}((\mathbf{a}^1\tau^1)^{r_1}\cdots(\mathbf{a}^n\tau^n)^{r_n})&=\sigma(\mathbf{a}^1\tau^1)^{r_1} \cdots \sigma(\mathbf{a}^n\tau^n)^{r_n}.
\end{align*}
We find
\[
\sigma(\forestGC)=1, \quad \sigma(\forestHC)=4, \quad \sigma(\forestIC)=2.
\]
We observe that $\sigma(\mathbf{a}\tau)=\sigma(\mathbf{a})\sigma(\tau)$ for $\mathbf{a}\in \AA$, $\tau\in T$, but we emphasize that $\sigma(\mathbf{a})$ is not the product of the symmetry coefficients of each aroma in $\mathbf{a}$.
Let the projection $\psi\colon \CC\FF\rightarrow \AA\FF$ be given by
\[
\psi\colon (\mathbf{a}^1\tau^1)\dots (\mathbf{a}^m\tau^m)\in \CC\FF \mapsto \mathbf{a}^1\dots \mathbf{a}^m\tau^1\dots \tau^m\in \AA\FF.
\]
Note that $\AA\FF$ cannot be injected straightforwardly to $\CC\FF$ as the multiaromas alone $\AA$ do not appear in $\CC\FF$.
Let the inner product $\langle \pi_1,\pi_2 \rangle=\sigma(\pi_1)\ind_{\pi_1=\pi_2}$ and the associated dual map of $\psi$:
\[
\psi^*\colon \AA\FF\rightarrow \CC\FF, \quad \langle \psi(\pi_1),\pi_2 \rangle=\langle \pi_1,\psi^*(\pi_2) \rangle.
\]
For instance, we find
\[
\psi^*(\forestJC)=(\forestKC)\ \forestLC+\forestMC\ (\forestNC).
\]

The pre-Lie grafting product $\curvearrowright$ extends to $\AA\FF$ and $\CC\FF$ by the standard Guin-Oudom procedure \cite{Oudom08otl}.
The Grossman-Larson product $\star$ is implicitly defined on both clumped and aromatic forests by
\[
(\pi_1\star\pi_2)\curvearrowright \pi_3 =\pi_1\curvearrowright (\pi_2\curvearrowright \pi_3).
\]
Then, $\AA\FF$ is equipped with the coproduct $\Delta a\pi=\Delta a\cdot \Delta \pi$, where we use the standard deshuffle coproduct on the symmetric algebras $\AA=S(A^0)$ and $S(T)$.
We find for instance:
\[
\Delta \forestOC=
(\forestPC\otimes\mathbf{1}
+\forestQC\otimes \forestRC
+\forestSC\otimes \forestTC
+\mathbf{1}\otimes \forestUC)\cdot
(\forestVC\otimes\mathbf{1}
+\forestWC\otimes \forestXC
+\forestYC\otimes \forestAD
+\mathbf{1}\otimes \forestBD).\]
On clumped forests, we consider the standard coproduct $\Delta$ on the symmetric algebras $\CC\FF$:
\[
\Delta (\forestCD)\cdot(\forestDD)=
(\forestED)\cdot(\forestFD)\otimes\mathbf{1}
+\forestGD\otimes \forestHD
+\forestID\otimes \forestJD
+\mathbf{1}\otimes (\forestKD)\cdot(\forestLD).
\]
We naturally obtain Hopf algebra structures, but also Hopf algebroid structures for the aromatic forests (that are out of the scope of the present paper, see \cite{Lu96haa, Moerdijk10otu}).
\begin{proposition}
The spaces $(\AA\FF,\cdot,\Delta)$, $(\AA\FF,\star,\Delta)$, $(\CC\FF,\cdot,\Delta)$, and $(\CC\FF,\star,\Delta)$ are Hopf algebras.
\end{proposition}

Let us extend the celebrated Butcher-Connes-Kreimer coproduct \cite{Connes98har, Kreimer99lfq} to $\AA\FF$ and $\CC\FF$.
\begin{definition}
An admissible cut $c$ of a tree $\tau\in T$ is a possibly empty subset of the edges of $\tau$ where any path from the root to a leaf of $x$ has at most one cut.
An admissible cut of an aroma $a=(\tau_1,\dots,\tau_K)\in A^0$, written as a tuple with cyclic invariance, is the union of admissible cuts of the $\tau_k$. In other words, an admissible cut of an aroma cannot cut the edges in the cycle.
If $c$ is admissible, the connected components obtained by removing the edges in $c$ are collected in $R^c(x)$ that contains the root or the cycle, and $P^c(x)$ for the other components.
\end{definition}

\begin{definition}
The Butcher-Connes-Kreimer coproduct on $A^0$ and $T$ is defined by the extraction contraction formula
\[
\Delta_{BCK}^{aro}x=x\otimes \mathbf{1}+\sum_{c\in Adm(x)} P^c(x)\otimes R^c(x),
\]
The coproduct is extended into $\Delta_{BCK}^{aro}\colon \AA\FF\rightarrow \AA\FF\otimes \AA\FF$ as a morphism:
\[
\Delta_{BCK}^{aro}a^1\cdots a^m\tau^1\cdots \tau^n=\Delta_{BCK}^{aro}a^1\cdots \Delta_{BCK}^{aro}a^m\cdot \Delta_{BCK}^{aro}\tau^1 \cdots \Delta_{BCK}^{aro}\tau^n.
\]
On clumped forests, the clumped BCK coproduct is defined for $\mathbf{a}\tau,\mathbf{a}_1\tau_1,\dots,\mathbf{a}_n\tau_n\in \AA T$ by
\[
\Delta_{BCK}^{cl}\tau=(\psi^*\otimes \ind)\circ \Delta_{BCK}^{aro}(\tau),\quad
\Delta_{BCK}^{cl}(\mathbf{a}_1\tau_1)\cdots (\mathbf{a}_n\tau_n)=\Delta_{BCK}^{cl}\mathbf{a}_1\tau_1\cdots \Delta_{BCK}^{cl}\mathbf{a}_n\tau_n.
\]
\end{definition}

\newcommand{\forestaroBCKexb}{
\tikz[planar forest ] {
\draw (-1.0,1.0) edge[-] (0.0,1.0);
\draw (0.0,1.0) edge[-] (1.0,1.0);
\draw (-1.0,1.0) arc (90:270:0.5);
\draw (1.0,1.0) arc (90:-90:0.5);
\draw (-1.0,0.0) edge[-] (1.0,0.0);
\node [draw=none] at (0.0, 0.0) { } 
child {node [b] at (-1.0, 1.0) {  } edge from parent[draw=none] 
child {node [b] at (0.0, 1.0) {  }  
child {node [b] at (0.0, 1.0) {  }  
}
}
}
child {node [b] at (0.0, 1.0) {  } edge from parent[draw=none] 
}
child {node [b] at (1.0, 1.0) {  } edge from parent[draw=none] 
child {node [b] at (0.0, 1.0) {  }  
}
}
;
\draw (-1.5, 1.5) -- (-0.5, 1.5);
}}

\newcommand{\forestaroBCKexc}{
\tikz[planar forest ] {
\draw (-1.0,1.0) edge[-] (0.0,1.0);
\draw (0.0,1.0) edge[-] (1.0,1.0);
\draw (-1.0,1.0) arc (90:270:0.5);
\draw (1.0,1.0) arc (90:-90:0.5);
\draw (-1.0,0.0) edge[-] (1.0,0.0);
\node [draw=none] at (0.0, 0.0) { } 
child {node [b] at (-1.0, 1.0) {  } edge from parent[draw=none] 
child {node [b] at (0.0, 1.0) {  }  
child {node [b] at (0.0, 1.0) {  }  
}
}
}
child {node [b] at (0.0, 1.0) {  } edge from parent[draw=none] 
}
child {node [b] at (1.0, 1.0) {  } edge from parent[draw=none] 
child {node [b] at (0.0, 1.0) {  }  
}
}
;
\draw (-1.5, 2.5) -- (-0.5, 2.5);
}}

\newcommand{\forestaroBCKexd}{
\tikz[planar forest ] {
\draw (-1.0,1.0) edge[-] (0.0,1.0);
\draw (0.0,1.0) edge[-] (1.0,1.0);
\draw (-1.0,1.0) arc (90:270:0.5);
\draw (1.0,1.0) arc (90:-90:0.5);
\draw (-1.0,0.0) edge[-] (1.0,0.0);
\node [draw=none] at (0.0, 0.0) { } 
child {node [b] at (-1.0, 1.0) {  } edge from parent[draw=none] 
child {node [b] at (0.0, 1.0) {  }  
child {node [b] at (0.0, 1.0) {  }  
}
}
}
child {node [b] at (0.0, 1.0) {  } edge from parent[draw=none] 
}
child {node [b] at (1.0, 1.0) {  } edge from parent[draw=none] 
child {node [b] at (0.0, 1.0) {  }  
}
}
;
\draw (0.5, 1.5) -- (1.5, 1.5);
}}

\newcommand{\forestaroBCKexe}{
\tikz[planar forest ] {
\draw (-1.0,1.0) edge[-] (0.0,1.0);
\draw (0.0,1.0) edge[-] (1.0,1.0);
\draw (-1.0,1.0) arc (90:270:0.5);
\draw (1.0,1.0) arc (90:-90:0.5);
\draw (-1.0,0.0) edge[-] (1.0,0.0);
\node [draw=none] at (0.0, 0.0) { } 
child {node [b] at (-1.0, 1.0) {  } edge from parent[draw=none] 
child {node [b] at (0.0, 1.0) {  }  
child {node [b] at (0.0, 1.0) {  }  
}
}
}
child {node [b] at (0.0, 1.0) {  } edge from parent[draw=none] 
}
child {node [b] at (1.0, 1.0) {  } edge from parent[draw=none] 
child {node [b] at (0.0, 1.0) {  }  
}
}
;
\draw (-1.5, 1.5) -- (-0.5, 1.5);
\draw (0.5, 1.5) -- (1.5, 1.5);
}}

\newcommand{\forestaroBCKexf}{
\tikz[planar forest ] {
\draw (-1.0,1.0) edge[-] (0.0,1.0);
\draw (0.0,1.0) edge[-] (1.0,1.0);
\draw (-1.0,1.0) arc (90:270:0.5);
\draw (1.0,1.0) arc (90:-90:0.5);
\draw (-1.0,0.0) edge[-] (1.0,0.0);
\node [draw=none] at (0.0, 0.0) { } 
child {node [b] at (-1.0, 1.0) {  } edge from parent[draw=none] 
child {node [b] at (0.0, 1.0) {  }  
child {node [b] at (0.0, 1.0) {  }  
}
}
}
child {node [b] at (0.0, 1.0) {  } edge from parent[draw=none] 
}
child {node [b] at (1.0, 1.0) {  } edge from parent[draw=none] 
child {node [b] at (0.0, 1.0) {  }  
}
}
;
\draw (-1.5, 2.5) -- (-0.5, 2.5);
\draw (0.5, 1.5) -- (1.5, 1.5);
}}

\begin{ex}
\label{ex:ex_aroBCK}
Consider the aromatic tree $\tau=\forestMD$, then the associated BCK coproduct is
\[
\Delta_{BCK}^{aro} \forestND=
((\forestOD\otimes \mathbf{1}+ \mathbf{1}\otimes \forestPD)
\cdot (\forestQD\otimes \mathbf{1}+ \forestRD\otimes \forestSD+ \mathbf{1}\otimes \forestTD))
\cdot (\forestUD\otimes \mathbf{1}+ \mathbf{1}\otimes \forestVD).
\]
From the clumped forest perspective, we remove the terms with aromas alone:
\[
\Delta_{BCK}^{cl} \forestWD=
\forestXD\otimes \mathbf{1}
+\forestYD\otimes \forestAE
+\forestBE\otimes \forestCE
+\mathbf{1}\otimes \forestDE.
\]
A more involved example is detailed in Table \ref{table:ex_aro_BCK} and further examples are presented in the appendices of \cite{Bogfjellmo19aso, Bronasco22cef}.

\begin{figure}[ht]
\begin{longtable}{|C|C|C|C|C|C|C|}
\hline
\text{Cut } c & \forestEE & \forestaroBCKexb & \forestaroBCKexc & \forestaroBCKexd & \forestaroBCKexe & \forestaroBCKexf \\\hline
P_c(\tau) & \mathbf{1} & \forestFE & \forestGE & \forestHE & \forestIE & \forestJE \\\hline
R_c(\tau) & \forestKE & \forestLE & \forestME & \forestNE & \forestOE & \forestPE \\\hline
\caption{Admissible cuts and related maps for an aroma.}
\label{table:ex_aro_BCK}
\end{longtable}
\end{figure}
\end{ex}

\begin{theorem}[\cite{Bogfjellmo19aso, Bronasco22cef}]
The space $(\AA\FF,\cdot,\Delta_{BCK}^{aro})$ is a Hopf algebra and is isomorphic (up to the symmetry coefficient) to the dual of the Grossman-Larson Hopf algebra $(\AA\FF,\star,\Delta)$, with $\Delta$ the standard deshuffle coproduct.
Analogously, the space $(\CC\FF,\cdot,\Delta_{BCK}^{cl})$ is a Hopf algebra and is isomorphic (up to the symmetry coefficient) to the dual of the Grossman-Larson Hopf algebra $(\CC\FF,\star,\Delta)$.
\end{theorem}

\begin{remark}
More generally, the aromatic/clumped forests can be equipped with a structure of pre-Hopf algebroid/algebra \cite{Bronasco22cef}. We mention that planar extensions of these objects exist and yield post-Hopf algebroid/algebra \cite{Li23pha, Busnot25pha, Busnot26tft}.
\end{remark}

\begin{remark}
The BCK structure allows to formulate the composition of differential operators given by a Taylor expansion (more precisely, by S-series). The product dual to the BCK coproduct is called the composition law and appears in the order analysis of numerical integrators of Runge-Kutta type.
This approach generalises to aromatic B-series methods \cite{MuntheKaas16abs, Bogfjellmo19aso, Laurent23tab} with the Hopf structures of aromatic and clumped forests.
\end{remark}

\section{Aromatic and clumped multi-indices}
\label{section:aromatic_multi-indices}

In this section, we introduce the new objects of aromatic and clumped multi-indices, generalising \cite{Lejay22cgr, Linares21tsg}, and we present the main results of this paper. More precisely, we uncover the pre-Lie-Rinehart and Hopf algebraic structures of the new multi-indices and derive Hopf embeddings from mutiindices to their tree counterparts, generalising \cite{Zhu24fna}.

\subsection{The free aromatic Novikov algebra}
\label{subsec:free_aromatic_Novikov_algebra}

A Novikov algebra is a vector space $N$ over a base field $\K$,  together with a bilinear product $\rhd: N\times N\rightarrow N$ such that, for any $x,y,z\in N$, the following identities hold:
\begin{align}
x\rhd (y\rhd z) - (x\rhd y)\rhd z & =y\rhd (x\rhd z) - (y \rhd x)\rhd z, \label{Novikov}\\
(x\rhd y)\rhd z & =(x\rhd z)\rhd y.\label{NAP}
\end{align}
The equation \eqref{Novikov} is the pre-Lie identity, while the equation \eqref{NAP}  is the right-NAP identity (for non associative permutative). We denote the free Novikov algebra generated by the set $C$ by $N(C)$.
This is described as follows: let $\overline{N}(C)$ be the commutative algebra of polynomials with variables $x_{j}^a, (a,j)\in C\times \{-1, 0, 1, 2, \cdots\}$, let $\partial$ be the unique derivation of $\overline{N}(C)$ such that $\partial x_j^a= x_{j+1}^a$.

A basis of $\overline{N}(C)$ is given by the monomials 
\[x^{\mathbf k}\colon= \prod_{j\geq -1, a\in C}(x_j^a)^{k_j^a},\]
and the weight of $x^{\mathbf k}$ is defined as
\[\mathbf {wt}(x^{\mathbf k})=\sum_{a\in C, j\leq -1 } j k_{j}^{a}\]
where the exponents $k_j^a$ are non-negative integers, and all terms of the product vanish except a finite number of them. The weight induces a unique $\Z$-grading of the algebra $\overline{N}(C)$, for which the derivation $\partial$ is homogeneous of degree one. The derivation $\partial$ satisfies
\begin{equation}\label{partialderivation}
    \partial x^{\mathbf k}=\sum_{j\geq -1, a\in C}k_j^ax^{\mathbf k - \mathbf e_j^a + \mathbf e_{j+1}^a}
\end{equation}
where $\mathbf e_j^a$ is the multi-index whose entries are all zero except for a single non-zero entry in position $(j,a)$. In other words, $x^{\mathbf e_j^a}=x_j^a$. Let $M_n$ be the monoid of weight $n$ and the bilinear product $P\rhd Q\colon=P\cdot\partial Q$. Then $(M_{-1},\rhd)$ is a pre-Lie algebra.

In the spirit of aromatic trees, we construct the aromatic Novikov algebra by $\AA M_{-1}:= S(M_0)\otimes M_{-1}$, where $S(V)$ denotes the symmetric tensorial algebra. Aromas are in $M_0$, multiaromas in $S(M_0)$, and $M_{-1}$ represents the analogue of trees.
The multi-indices analogues of aromatic and clumped forests are $\AA\MM=S(M_0)\otimes S(M_{-1})$ and $\CC\MM=S(\AA M_{-1})$.
For the sake of simplicity, we use the same notation $\odot$ for the products in $S(M_0)$, in $S(M_{-1})$, and in $\AA\MM$. Moreover, the unit of all symmetric tensorial algebras is denoted as $\mathbf{1}$.
The product of the symmetric algebra $\CC\MM=S(\AA M_{-1})$ is denoted by $\diamond$.
Aromatic monomials $\AA\MM$ are spanned by elements of the form
$$x^{\kappa^1}\odot \cdots \odot x^{\kappa^m} \odot x^{\mathbf k^1}\odot \cdots \odot x^{\mathbf k^n},\quad x^{\kappa^i}\in M_0,\quad x^{\mathbf k^i}\in M_{-1}.$$
On the other hand, clumped monomials are spanned by elements of the form
$$(\mathbf y^1\odot x^{\mathbf k^1})\diamond \cdots \diamond (\mathbf y^n\odot x^{\mathbf k^n}),\quad \mathbf y^i\in S(M_0),\quad x^{\mathbf k^i}\in M_{-1}.$$

The weight is defined on aromatic monomials by
\[\mathbf {wt}(x^{\kappa^1}\odot \cdots \odot x^{\kappa^m} \odot x^{\mathbf k^1}\odot \cdots \odot x^{\mathbf k^n})
=\sum_{i=1}^m \mathbf {wt}(x^{\kappa^i}) +\sum_{i=1}^n \mathbf {wt}(x^{\mathbf k^i}),\]
and is then extended on clumped monomials by
\[\mathbf {wt}((\mathbf y^1\odot x^{\mathbf k^1})\diamond \cdots \diamond (\mathbf y^n\odot x^{\mathbf k^n}))
=\sum_{i=1}^n \mathbf {wt}(\mathbf y^i\odot x^{\mathbf k^i}).\]
Note that the weight on monomials in $\AA M_{-1}$ is $-1$ as expected.
The symmetry factor on $\AA \MM$ is given by,
\[\sigma ((x^{\kappa^1})^{\odot q_1}\odot \cdots \odot (x^{\kappa^m})^{\odot q_m} \odot (x^{\mathbf k^1})^{\odot p_1}\odot \cdots \odot (x^{\mathbf k^n})^{\odot p_n})
=\prod_{j=1}^m q_j!(\kappa^j!)^{q_j}\prod_{i=1}^n p_i!(\mathbf k^i!)^{p_i}.\]
Then, we have on clumped monomials
\[\sigma ((\mathbf y^1\odot x^{\mathbf k^1})^{\diamond p_1}\diamond \cdots \diamond (\mathbf y^n\odot x^{\mathbf k^n})^{\diamond p_n})
=\prod_{i=1}^n p_i!\sigma (\mathbf y^i\odot x^{\mathbf k^i})^{p_i}.\]
The degree of an element of $\AA \MM$ is
\[
\deg(x^{\kappa^1}\odot \cdots \odot x^{\kappa^m} \odot x^{\mathbf k^1}\odot \cdots \odot x^{\mathbf k^n})
=\sum_{i=1}^m \sum_{a\in C, j\geq -1}\kappa_{j}^{i,a} +\sum_{i=1}^n \sum_{a\in C, j\geq -1}\mathbf k_{j}^{i,a},\]
and analogously on $\CC\MM$.
The pairing on $\AA \MM$ (respectively $\CC\MM$) is given by
\[\langle\mathbb M, \mathbb M'\rangle:= \sigma(\mathbb M)\ind_{\mathbb M=\mathbb M'},\]
where we use the symmetry coefficient on $\AA\MM$ (respectively $\CC\MM$).

Let us now extend the product $\rhd$.
The bilinear product $\rhd \colon \AA M_{-1}\times M_n$ for $n=-1,0$ is defined by $(\mathbf y \odot x^{\mathbf k})\rhd P =\mathbf y \odot (x^{\mathbf k}\cdot \partial P)$.
Then, $\rhd$ extends on $S(M_0)$ and $\AA M_{-1}$ by the Leibniz rule:
\begin{align*}
    (\mathbf y^1 \odot x^{\mathbf k})\rhd (x^{\kappa^1}\odot \cdots \odot x^{\kappa^m}) &=
    \sum_{j=1}^m \mathbf y^1\odot x^{\kappa}\odot \cdots \odot (x^{\mathbf k^1} \rhd x^{\kappa^j}) \odot \cdots \odot x^{\kappa^m}\\
    (\mathbf y^1 \odot x^{\mathbf k^1})\rhd (\mathbf y^2 \odot x^{\mathbf k^2}) 
    &= \mathbf y^1\odot (x^{\mathbf k^1} \rhd \mathbf y^2) \odot x^{\mathbf k^2}
    +\mathbf y^1\odot \mathbf y^2 \odot (x^{\mathbf k^1} \rhd x^{\mathbf k^2})
\end{align*}
As $\rhd$ is defined as a derivation, the aromatic Novikov algebra $(\AA M_{-1},\rhd)$ naturally satisfies the identities \eqref{Novikov} and \eqref{NAP}, and thus is a pre-Lie algebra.
More precisely, we show that the aromatic Novikov algebra $\AA M_{-1}$ is equipped with a pre-Lie-Rinehart algebra structure.
\begin{proposition}\label{prelierinehart}
The aromatic Novikov algebra $\AA M_{-1}=S(M_0)\otimes M_{-1}$ is a pre-Lie-Rinehart algebra for the product
\begin{eqnarray*}
\AA M_{-1} \times \AA M_{-1} & \longrightarrow & \AA M_{-1}\\
 (\mathbf y^1 \odot x^{\mathbf k^1}, \mathbf y^2 \odot x^{\mathbf k^2})& \longmapsto  & (\mathbf y^1 \odot x^{\mathbf k^1})\rhd (\mathbf y^2 \odot x^{\mathbf k^2})
\end{eqnarray*}
and the induced maps:
\begin{eqnarray*}
& \nabla:\AA M_{-1} \pil \Hom_{\mathbf k}(\AA M_{-1}, \AA M_{1}) \\
& \mathbf y^1 \odot x^{\mathbf k^1} \mapsto \big(\mathbf y^2 \odot x^{\mathbf k^2} \mapsto (\mathbf y^1 \odot x^{\mathbf k^1}) \rhd( \mathbf y^2 \odot x^{\mathbf k^2}) \big)\\
\\
& \rho:\AA M_{-1} \pil \Der_{\mathbf k}(S(M_0), S(M_0))\\
& \mathbf y^1 \odot x^{\mathbf k^1}\mapsto \big(\mathbf y^2 \mapsto (\mathbf y^1 \odot x^{\mathbf k^1}) \rhd \mathbf y^2\big)
\end{eqnarray*}
where the space $\Der_{\mathbf k}(S(M_0), S(M_0))$ collects all the derivation maps on $S(M_0)$ and $\AA M_{-1}$ is a Lie algebra for the Lie bracket $[x, y]=x\rhd y - y\rhd x$.
\end{proposition}

\begin{proof}
Let us check that the map $\rho$ satisfies the Leibniz rule: for $\mathbf y^1 \odot x^{\mathbf k^1}$, $\mathbf y^2 \odot x^{\mathbf k^2}\in \AA M_{-1}$ and $\mathbf y\in S(M_0)$, we have
\begin{align*}
[\mathbf y^1 \odot x^{\mathbf k^1}, \mathbf y \odot \mathbf y^2 \odot x^{\mathbf k^2}]
& = (\mathbf y^1 \odot x^{\mathbf k^1}) \rhd (\mathbf y \odot \mathbf y^2 \odot x^{\mathbf k^2}) - (\mathbf y \odot \mathbf y^2 \odot x^{\mathbf k^2}) \rhd (\mathbf y^1 \odot x^{\mathbf k^1})\\
& = \mathbf y \odot \mathbf y^1 \odot x^{\mathbf k^1} \odot \partial(\mathbf y^2 \odot x^{\mathbf k^2})
- \mathbf y \odot \mathbf y^2 \odot x^{\mathbf k^2} \odot \partial(\mathbf y^1 \odot x^{\mathbf k^1})\\
&+\mathbf y^1 \odot x^{\mathbf k^1} \odot \partial(\mathbf y) \odot \mathbf y^2 \odot x^{\mathbf k^2}\\
& =  \mathbf y\odot [\mathbf y^1 \odot x^{\mathbf k^1}, \mathbf y^2 \odot x^{\mathbf k^2}] + \rho(\mathbf y^1 \odot x^{\mathbf k^1})(\mathbf y) \mathbf y^2 \odot x^{\mathbf k^2}.
\end{align*}
One shows similarly that $[-, -]$ is a Lie bracket.
\end{proof}

\begin{remark}
Following \cite{Grong23pla}, the product $\rhd$ can be seen as a connection on $\AA M_{-1}$ with vanishing torsion and curvature. The multiaromas represent the functions, while the aromatic Novikov algebra contains the analogues of vector fields.
\end{remark}

\subsection{The Hopf algebras of aromatic and clumped monomials}
\label{aroclump}

Our aim is to extend the LOT Hopf algebra \cite{Linares21tsg, Zhu24fna} to the aromatic monomials. We recall that the standard LOT Hopf algebra generated by a set $C$ is $\HH_{LOT}^{C}=\big(S(N(C)), \star, \Delta, \ind, \epsilon\big)^\circ$, that is, the graded dual of the Hopf algebra $\big(S(N(C)), \star, \Delta, \ind, \epsilon \big)$, equipped with the Grossman-Larson product $\star$ and the deshuffle coproduct $\Delta$. This structure is analogous to the BCK Hopf algebra of Butcher trees \cite{Connes98har} for multi-indices. We extend here the LOT structure to the aromatic context.

We provide a definition of the aromatic LOT coproduct using free edges and cuts in Section \ref{section:Hopf_structure}.
We will then provide a straightforward explicit formula for the coproduct, that we give here directly for the sake of simplicity. The proof of this theorem is postponed to Section \ref{explicit formula}.
\begin{theorem} \label{arobck1}
The aromatic LOT coproduct $\Delta_{LOT}^{aro}: \AA M_{-1} \rightarrow \AA\MM\otimes \AA M_{-1}$ is the unique unital morphism defined by
\begin{equation}\label{coproduct}
\Delta_{LOT}^{aro}(\mathbf y\odot x^{\mathbf k})  = \Delta_{LOT}^{aro}(x^{\kappa_1})\odot \cdots\odot \Delta_{LOT}^{aro}(x^{\kappa_m})\odot\Delta_{LOT}^{aro}(x^{\mathbf k})
\end{equation}
where $\mathbf y=x^{\kappa_1}\odot\cdots\odot x^{\kappa_m}$, the coproduct on aromas $M_0$ and monomials $M_{-1}$ is
\begin{align*}
\Delta_{LOT}^{aro}(x^\kappa) & = \sum_{t\geq 0}\sum_{\underset{\mathbf {wt}(x^{\kappa^i})=-1}{\kappa=\kappa^1 + \cdots + \kappa^t + \overline{\kappa}}} \frac{\kappa!}{\sigma(x^{\kappa^1}\odot \cdots\odot x^{\kappa^t})\overline{\kappa}!} x^{\kappa^1}\odot \cdots\odot  x^{\kappa^t}\otimes \overline{\partial}^t x^{\overline{\kappa}} + x^{\kappa}\otimes \mathbf 1,\\
\Delta_{LOT}^{aro}(x^{\mathbf k})& =   \sum_{r\geq 0} \sum_{\underset{\mathbf {wt}(x^{\mathbf k^i})=-1}{\mathbf k = \mathbf k^1 + \cdots + \mathbf k^r + \overline{\mathbf k}}} \frac{\mathbf k!}{\sigma(x^{\mathbf k^1}\odot \cdots \odot  x^{\mathbf k^r})\overline{\mathbf k}!}x^{\mathbf k^1}\odot \cdots\odot  x^{\mathbf k^r}\otimes \overline{\partial}^r x^{\overline{\mathbf k}} + x^{\mathbf k}\otimes \mathbf 1,    
\end{align*}
and $\overline{\partial}$ is explictly given in Proposition \ref{prop:formula_partial_bar}.
\end{theorem}

Let us consider the projection $\varphi: \CC\MM\rightarrow \AA\MM$, defined by
\[\varphi: (\mathbf y^1\odot x^{\mathbf k^1})\diamond\cdots\diamond (\mathbf y^n\odot x^{\mathbf k^n})\in \CC\MM \longmapsto \mathbf y^1\odot \cdots \odot \mathbf y^n \odot x^{\mathbf k^1}\odot \cdots \odot x^{\mathbf k^n}\in \AA\MM.\]
Since the multi-aromas alone $S(M_0)$ do not appear in $\CC\MM$, the algebra $\AA\MM$ does not inject straightforwardly into $\CC\MM$.
Let $\varphi^*: \AA\MM\rightarrow \CC\MM$ be the dual of $\varphi$ for the inner product $\langle -,-\rangle$, that is,
\[\langle \varphi(\M), \M'\rangle = \langle \M, \varphi^* (\M') \rangle.\]
The LOT coproduct on aromatic monomials extends straightforwardly as a morphism for $\diamond$ as
$\Delta_{LOT}^{aro}\colon \CC\MM\rightarrow \AA\MM\otimes \CC\MM$.
Then, the LOT coproduct on clumped monomials is for $\pi_{cl}\in \CC\MM$:
\[\Delta_{LOT}^{cl}=(\varphi^* \otimes \ind)\circ \Delta_{LOT}^{aro}.\]
% This equips $\HH_{LOT}^{cl}=(\CC\MM, \odot, \Delta_{LOT}^{cl})$\moda{ with ?}.

The LOT structures on aromatic and clumped monomials yield analogous Hopf structures to the ones of aromatic forests $\AA\FF$ and clumped forests $\CC\FF$. The proof is analogous to the one on tree structures from \cite{Bronasco22cef}.
\begin{proposition}
The aromatic multi-indices $(\AA\MM,\odot, \Delta_{LOT}^{aro})$ and the clumped multi-indices $(\CC\MM, \diamond, \Delta_{LOT}^{cl})$ equipped with the LOT coproduct are Hopf algebras.
\end{proposition}

\subsection{Hopf embedding between multi-indices and forests}

In this section, we present the main result of this paper, that is, that there exist Hopf embeddings $j^{aro}$ and $j^{cl}$ from aromatic/clumped monomials into their trees counterpart.

The Hopf embeddings are based on the fertility map $\Phi$, which is defined as follows.
\begin{definition}\label{fertilitymap}
Let the fertility map $\Phi$ be given by
\begin{align*}
\Phi: \AA\FF &\rightarrow \AA \MM\\
a^1\cdots a^n t &\mapsto \Phi(a^1)\odot \cdots \odot \Phi(a^n)\odot \Phi(t)
\end{align*}
where
\[\Phi(t)=\prod_{v\in V(t)}x_{f(v) - 1}^{d(v)},\quad
\Phi(a^j)=\prod_{v\in V(a^j)}x_{f(v) - 1}^{d(v)},\]
and $d(v)$ and $f(v)$ are the colour and the fertility of the vertex $v$. 
The fertility map $\Phi$ is defined on $\CC\FF$ as a morphism:
\[
\Phi\big((\mathbf{a}^1\tau^1)\cdots (\mathbf{a}^n\tau^n)\big)=\Phi(\mathbf{a}^1\tau^1)\diamond \cdots \diamond \Phi(\mathbf{a}^n\tau^n).
\]
\end{definition}

\begin{theorem}
The fertility map induces a pre-Lie-Rinehart morphism $\Phi: \AA\TT \rightarrow \AA M_{-1}$.
\end{theorem}

\begin{proof}
By construction, the map $\Phi: (\AA\TT, \curvearrowright) \rightarrow (\AA M_{-1},\rhd)$ is a pre-Lie algebra morphism.
Then, the map $\Phi$ satisfies
\begin{align*}
\Phi([\mathbf{a}^1t^1, \mathbf{a}\mathbf{a}^2 t^2])
&= \Phi(\mathbf{a}^1t^1\curvearrowright \mathbf{a}\mathbf{a}^2 t^2-\mathbf{a}\mathbf{a}^2 t^2\curvearrowright \mathbf{a}^1t^1)\\
&= \Phi(\mathbf{a}^1t^1)\rhd (\Phi(\mathbf{a})\Phi(\mathbf{a}^2 t^2))-\Phi(\mathbf{a})\Phi(\mathbf{a}^2 t^2)\rhd \Phi(\mathbf{a}^1t^1)\\
&= \Phi(\mathbf{a}) [\Phi(\mathbf{a}^1t^1), \Phi(\mathbf{a}^2 t^2)]
+\rho(\Phi(\mathbf{a}^1t^1))(\Phi(\mathbf{a})) \Phi(\mathbf{a}^2 t^2)).
\end{align*}
Hence the result.
\end{proof}

\begin{definition}
Let $j^{aro}$ be the dual map to $\Phi$ on $M_0$ and $M_{-1}$, that is,
\[
\langle \Phi(a), x^\kappa \rangle = \langle a, j^{aro}( x^\kappa ) \rangle,\quad
\langle \Phi(t), \mathbf x^{\mathbf k} \rangle = \langle t, j^{aro}( x^{\mathbf k}) \rangle.
\]
We extend $j^{aro}\colon \AA\MM\rightarrow\AA\FF$ as a morphism
\[
j^{aro}\big(x^{\kappa^1}\odot\cdots \odot x^{\kappa^m}\odot x^{\mathbf k^1}\odot\cdots \odot x^{\mathbf k^n}\big) = j^{aro}(x^{\kappa^1})\cdots j^{aro}(x^{\kappa^m}) j^{aro}(x^{\mathbf k^1}) \cdots j^{aro}(x^{\mathbf k^n}).
\]
The analogue $j^{cl}\colon \CC\MM\rightarrow\CC\FF$ on clumped monomials is
\[
j^{cl}\big((\mathbf{y}^1\odot  x^{\mathbf k^1})\diamond\cdots \diamond (\mathbf{y}^n\odot x^{\mathbf k^n})\big)=j^{aro}(\mathbf{y}^1 \odot x^{\mathbf k^1})\cdots j^{aro}(\mathbf{y}^n\odot x^{\mathbf k^n}). 
\]
\end{definition}

The following formula is deduced straightforwardly from the definition. Note that it does not extend in general to $\AA\MM$.
\begin{lemma}
The maps $j^{aro}$ satisfies
\begin{equation}\label{formulajaro}
j^{aro}(x^{\kappa}) = \sum_{\Phi(a)= x^{\kappa}} \frac{\sigma(x^{\kappa})}{\sigma(a)}a,\quad
j^{aro}(x^{\mathbf{k}}) = \sum_{\Phi(t)= x^{\mathbf{k}}} \frac{\sigma(x^{\mathbf{k}})}{\sigma(t)}t.
\end{equation}
\end{lemma}

\newcommand{\chetree}{
\tikz[planar forest ] {

\node [b] at (0.0, 0.0) {  } 
child {node [b] at (-0.5, 1.0) {  }  
}
child {node [b] at (0.5, 1.0) {  }  
}
;
}}

\newcommand{\ladrtree}{
\tikz[planar forest ] {

\node [b] at (0.0, 0.0) {  } 
child {node [b] at (0.0, 1.0) {  }  
}
;
}}

\newcommand{\chetreeplus}{
\tikz[planar forest ] {

\node [b] at (0.0, 0.0) {  } 
child {node [b] at (-0.5, 1.0) {  }  
}
child {node [b] at (0.5, 1.0) {  }  
child {node [b] at (0.0, 1.0) {  }  
}
}
;
}}

\newcommand{\chetreeY}{
\tikz[planar forest ] {

\node [b] at (0.0, 0.0) {  } 
child {node [b] at (0.0, 1.0) {  } 
child {node [b] at (-0.5, 1.0) {  }  }
child {node [b] at (0.5, 1.0) {  }  }
}
;
}}

\newcommand{\twolooptree}{
\tikz[planar forest ] {

\draw (-1.0,1.0) edge[-] (0.0,1.0);
\draw (-1.0,1.0) arc (90:270:0.5);
\draw (0.0,1.0) arc (90:-90:0.5);
\draw (-1.0,0.0) edge[-] (0.0,0.0);
\node [draw=none] at (0.0, 0.0) { } 
child {node [b] at (-1.0, 1.0) {  } edge from parent[draw=none] 
child {node [b] at (0.0, 1.0) {  } }
}
child {node [b] at (0.0, 1.0) {  } edge from parent[draw=none] 
child {node [b] at (0.0, 1.0) {  }}
}
;
}}

\newcommand{\twolooptreeV}{
\tikz[planar forest ] {

\draw (-1.0,1.0) edge[-] (0.0,1.0);
\draw (-1.0,1.0) arc (90:270:0.5);
\draw (0.0,1.0) arc (90:-90:0.5);
\draw (-1.0,0.0) edge[-] (0.0,0.0);
\node [draw=none] at (0.0, 0.0) { } 
child {node [b] at (-1.0, 1.0) {  } edge from parent[draw=none] 
}
child {node [b] at (0.0, 1.0) {  } edge from parent[draw=none] 
child {node [b] at (-0.5, 1.0) {  }}
child {node [b] at (0.5, 1.0) {}}
}
;
}}

\newcommand{\onelooptreeY}{
\tikz[planar forest ] {

\draw (-1.0,1.0) arc (90:270:0.5);
\draw (-1.0,1.0) arc (90:-90:0.5);
\node [draw=none] at (0.0, 0.0) { } 
child {node [b] at (-1.0, 1.0) {  } edge from parent[draw=none] 
child {node [b] at (0.0, 1.0) {  }
child {node [b] at (-0.5, 1.0) {  }}
child {node [b] at (0.5, 1.0) {}}
}
}
;
}}

\newcommand{\twolooptreeY}{
\tikz[planar forest ] {

\draw (-1.0,1.0) edge[-] (0.0,1.0);
\draw (-1.0,1.0) arc (90:270:0.5);
\draw (0.0,1.0) arc (90:-90:0.5);
\draw (-1.0,0.0) edge[-] (0.0,0.0);
\node [draw=none] at (0.0, 0.0) { } 
child {node [b] at (-1.0, 1.0) {  } edge from parent[draw=none] 
}
child {node [b] at (0.0, 1.0) {  } edge from parent[draw=none] 
child {node [b] at (0.0, 1.0) {}
child {node [b] at (-0.5, 1.0){}}
child {node [b] at (0.5, 1.0){}}
}
}
;
}}

\newcommand{\twolooptreeLadder}{
\tikz[planar forest ] {

\draw (-1.0,1.0) edge[-] (0.0,1.0);
\draw (-1.0,1.0) arc (90:270:0.5);
\draw (0.0,1.0) arc (90:-90:0.5);
\draw (-1.0,0.0) edge[-] (0.0,0.0);
\node [draw=none] at (0.0, 0.0) { } 
child {node [b] at (-1.0, 1.0) {  } edge from parent[draw=none]
child {node [b] at (0.0, 1.0){}}
}
child {node [b] at (0.0, 1.0) {  } edge from parent[draw=none] 
child {node [b] at (0.0, 1.0) {}
child {node [b] at (0.0, 1.0){}}
}
}
;
}}

For example, one finds:
\begin{align*}
j^{aro}(x_0x_{-1})=\ladrtree, \quad j^{aro}(x_{-1}^2x_{1})=\chetree, \quad j(x_{-1}^2x_0x_{1})=2\chetreeplus + \chetreeY,
\end{align*}
and for aromas,
\begin{align*}
j(x_{-1}^2x_1^2x_0)=2\, \twolooptreeY + 4\,\twolooptreeLadder, \quad j(x_{-1}^2x_1^2)=4\,\twolooptree + 2\,\onelooptreeY.
\end{align*}

\begin{definition}
The symmetry factor of aroma forest is the cardinal of its automorphism group, then for multi-aromas $\mathbf a=(a^1)^{\odot \ell_1}\odot \cdots \odot (a^n)^{\odot \ell_n}$, where $a^1, \ldots, a^n$ are pairwise distinct aromas, the symmetry factor can be expressed by its distinct aroma components
\begin{align*}
\sigma(a) = \ell_1!\cdots \ell_n!\sigma(a^1)^{\ell_1}\cdots \sigma(a^n)^{\ell_n},
\end{align*}
here the symmetry factor $\sigma(a)$ is obtained by external symmetry factor $\sigma^{ext}(\mathbf a):=\ell_1!\cdots\ell_n!$ by the internal symmetry factor $\sigma^{int}(a)=\sigma(a^1)^{\ell_1}\cdots \sigma(a^n)^{\ell_n}$, futhermore for an element $\mathbf a \mathbf \tau \in \AA\FF$ with $\mathbf a=(a^{1})^{\odot \ell_1}\odot \cdots \odot (a^n)^{\ell_n}$ and $\tau=(\tau^1)^{\odot q_1}\odot \cdots \odot (\tau^m)^{\odot q_m}$. The external symmetry factor is $\sigma^{ext}(\mathbf a\tau)=\ell_1!\cdots \ell_n!q_1!\cdots q_m!$ and internal symmetry factor $\sigma^{int}(\mathbf a\tau)=\sigma(a^1)^{\ell_1}\cdots \sigma(a^n)^{\ell_n}\sigma(\tau^1)^{q_1}\cdots \sigma(\tau^m)^{q_m}.$
\end{definition}

The maps $j^{aro}$ and $j^{cl}$ are Hopf embeddings from the multi-indices algebras to their tree counterparts, as shown in the following result, whose proof is postponed to Section \ref{section:Hopf_structure}.
\begin{theorem}
\label{thm:embedding}
The map $j^{aro}$ (respectively $j^{cl}$) yields the following Hopf algebroid embedding (respectively Hopf algebra embedding):
\begin{align*}
    j^{aro}&\colon \HH_{LOT}^{aro}=\big(\AA\MM, \cdot, \Delta_{LOT}^{aro}\big) \rightarrow \HH_{BCK}^{aro}=\big(\AA\FF, \cdot, \Delta_{BCK}^{aro}\big),\\
    j^{cl}&\colon \HH_{LOT}^{cl}=\big(\CC\MM, \cdot, \Delta_{LOT}^{cl}\big) \rightarrow \HH_{BCK}^{cl}=\big(\CC\FF, \cdot, \Delta_{BCK}^{cl}\big).
\end{align*}
In addition, the following diagram commutes.
\[
\begin{tikzcd}
\HH_{LOT}^{aro} \arrow[r, "j^{aro}",hook] \arrow[d,"\varphi^*"' ]& \HH_{BCK}^{aro} \arrow[d,"\psi^*" ]\\
\HH_{LOT}^{cl} \arrow[r, "j^{cl}",hook]  & \HH_{BCK}^{cl}
\end{tikzcd}
\]
\end{theorem}

\subsection{Aromatic multi-indices for numerical volume-preservation}
\label{sec:num}

The design of numerical integrators that preserve volume is strongly linked to the characterisation of the kernel of the divergence map $\Ker(d)$ on aromatic trees $\AA T$ (see Definition \ref{def:div}) via backward error analysis \cite{Hairer06gni}.
The characterisation of $\Ker(d)$ is described in \cite{Laurent23tab, Laurent23tld, Dotsenko24vpo} with the so-called aromatic bicomplex. The resulting conditions for volume-preservation are numerous and challenging, so that one would be interested in a simpler case to find insight on the form of a volume-preserving aromatic B-series method.
For instance, the paper \cite{Bogfjellmo22uat} focuses on polynomial vector fields.
One could tackle the problem in dimension one, which would yield (standard) multi-indices, but numerical volume-preservation becomes a trivial problem and so does the algebra. We also recall the result from \cite{Bruned25edf} that states that there is no intermediate algebraic formalism between multi-indices (one dimensional) and Butcher trees (infinite dimensional).
We observe in this subsection that aromatic multi-indices provide an intermediate between multi-indices and aromatic trees for formulating a simplified but non-trivial set of necessary conditions for volume-preservation.

\begin{definition}
Let the divergence of multi-indices $d\colon \AA M_{-1}\rightarrow \AA$ be given by
\[
d\big(x^{\kappa^1}\odot\cdots \odot x^{\kappa^m}\odot x^{\mathbf k}\big)
=x^{\kappa^1}\odot\cdots \odot x^{\kappa^m}\odot \partial x^{\mathbf k}
+\sum_{i=1}^m x^{\kappa^1}\odot\cdots \odot (x^{\mathbf k}\rhd x^{\kappa^i}) \odot\cdots \odot x^{\kappa^m}.
\]
\end{definition}

The divergence on aromatic trees and aromatic multi-indices relate through the fertility map, as given by the following result whose proof is very similar to the one of Proposition \ref{prop:fertility_compatibility}.
\begin{lemma}
The divergence and the fertility map commute:
\[d\circ \Phi=\Phi\circ d.\]
\end{lemma}

The search for volume-preserving integrators in dimension one not only is a trivial problem in numerics, but also fully trivialises the algebraic conditions for volume-preservation, as explained in the following result.
\begin{proposition}
Let $\pi\colon \AA M_{-1}\rightarrow M_{-1}$ be the projection
\[
\pi\big(x^{\kappa^1}\odot\cdots \odot x^{\kappa^m}\odot x^{\mathbf k}\big)
=x^{\kappa^1}\cdots x^{\kappa^m}\cdot x^{\mathbf k}.
\]
Then, the kernel of the divergence is trivial when rewritten with standard multi-indices:
\[\pi\circ \Phi(\Ker(d))=0.\]
\end{proposition}

\begin{proof}
Following \cite{Laurent23tab}, $\Ker(d)$ is exactly generated by the elements of the following form (id est, the exterior derivative of aromatic 2-forms $a\tau_1\wedge \tau_2$),
\[
\omega=\Big((\tau_2\curvearrowright a) \tau_1-(\tau_1\curvearrowright a) \tau_2\Big)
+a\Big(\tau_2\curvearrowright\tau_1-(d\tau_1)\tau_2\Big)
-a\Big(\tau_1\curvearrowright\tau_2-(d\tau_2)\tau_1\Big).
\]
The map $\pi\circ\Phi$ vanishes on each of the three terms in $\omega$.
\end{proof}

The space of aromatic multi-indices brings an intermediate object which is simpler than aromatic trees and where much of the conditions for volume-preservation are retained.
A weaker formulation of volume preservation then is to find numerical methods whose modified vector field writes as a formal series indexed by $\omega \in \AA T$ (or its completion to be exact) and such that $\Phi(\omega)\in\Ker(d)$.
We present the first generators of $\Ker(d)$ for aromatic trees and aromatic multi-indices in Table \ref{table:Ker_d} in order to show how richer the structure of aromatic multi-indices is compared to standard multi-indices.
We emphasize that $\Phi$ is not injective on $\Ker(d)$, so that one would indeed obtain less restrictive conditions for volume-preservation using aromatic multi-indices.

\begin{figure}[ht]
\begin{longtable}{|C|C|C|}
\hline
\omega\in \Ker(d) & \Phi(\omega) \\\hline
\forestQE+\forestRE-\forestSE-\forestTE
& x_{-1}x_1\odot x_{-1}+x_0^2\odot x_{-1}-x_{-1}^2x_1-x_0\odot x_{-1}x_0\\
\hline
\forestUE+\forestVE+\forestWE
& 2x_{-1}x_0 x_1\odot x_{-1}+x_0^3\odot x_{-1}\\
-\forestXE-\forestYE-\forestAF
& -2x_{-1}^2x_0x_1-x_0\odot x_{-1}x_0^2\\
\hline
\forestBF+2\forestCF+\forestDF
& x_{-1}^2 x_2\odot x_{-1}+2x_{-1}x_0 x_1\odot x_{-1}\\
-2\forestEF-\forestFF-\forestGF
&-x_{-1}^2x_0x_1-x_{-1}^3 x_2-x_0\odot x_{-1}^2x_1\\
\hline
\forestHF+\forestIF+\forestJF
& x_0\odot x_{-1}x_1\odot x_{-1}+x_0\odot x_0^2\odot x_{-1}+x_{-1}x_0x_1\odot x_{-1}\\
-\forestKF-\forestLF-\forestMF
& -x_0\odot x_{-1}^2x_1-x_0\odot x_0\odot x_{-1}x_0-x_{-1}x_1\odot x_{-1}x_0\\
\hline
\caption{Generators of the kernel of the divergence map in the context of aromatic trees and their associated aromatic multi-indices for order up to four.}
\label{table:Ker_d}
\end{longtable}
\end{figure}

\section{Free edges, graftings, cuts, and Hopf embeddings}
\label{section:Hopf_structure}

This section defines the necessary concepts (free edges, graftings,...) and use them for proving the Hopf embeddings of Theorem \ref{arobck1} and the explicit description of the aromatic LOT coproduct.

\subsection{Free edges and extended maps}

Let us extend the aromatic trees and aromatic multi-indices to include free edges and higher weight.
\begin{definition}
Let $\overline{A_0}$ and $\overline{T}$ be the spaces of aromas and trees with free edges given by
\[\overline{A_0}=\bigoplus_{n=0}^{\infty}A_0^{(n)}, \quad \overline{T} = \bigoplus_{n}^{\infty} T^{(n)},\]
where $A_0^{(n)}$ and $T^{(n)}$ are the spaces of aromas with $n$ free edges and space of rooted trees with $n$ free edges. Then we define the space of aromatic trees with free edges by
\[\overline{\AA T}=S(\overline{A_0})\otimes \overline{T}.\]
The weight is extended as the total number of the edges minus the number of vertices for tree-like structures.
Analogously, let the aromatic multi-indices with higher weight,
\[\overline{M_{-1}}=\bigoplus_{n=0}^{\infty}M_{-1}^{(n)}, \quad \overline{M_0}=\bigoplus_{n=0}^{\infty}M_0^{(n)},\]
where $M_p^{(n)}\simeq M_{n+p}$ is the space of weight $n+p$ with respect to the original space $M_p$ and $p=0,1$. Furthermore, we define the aromatic Novikov algebra with higher weight:
\[\overline{\AA M_{-1}}=S(\overline{M_0})\otimes \overline{M_{-1}}.\]
\end{definition}

\newcommand{\onearomaFreeex}{
\tikz[planar forest ] {

\draw (-0.5, 4.0) edge[-] (-1.0, 3.0);
\draw (0.5, 3.0) edge[-] (0.0, 2.0);
\draw (-1.0,1.0) arc (90:270:0.5);
\draw (-1.0,1.0) arc (90:-90:0.5);
\node [draw=none] at (0.0, 0.0) { } 
child {node [b] at (-1.0, 1.0) {  } edge from parent[draw=none] 
child {node [b] at (1.0, 1.0) {}}
child {node [b] at (0.0, 1.0) {}
child {node [b] at (0.0, 1.0){}}
}
}
;
}}

\newcommand{\onearomarootedFreeex}{
\tikz[planar forest ] {

\draw (0.5, 3.0) edge[-] (0.0, 2.0);
\draw (1.5, 2.0) edge[-] (1.0, 1.0);
\draw (-0.5, 1.0) edge[-] (0.0, 0.0);
\node [draw=none] at (0.0, 0.0) { } ;
\node [b] at (0.0, 0.0) {}
child {node [b] at (0.0, 1.0) {}
child {node [b] at (0.0, 1.0) {}}
}
child {node [b] at (1.0, 1.0) {}}
;
}}

\begin{ex}
The monomials $x_0^3x_2\in \overline{M_0}$ and $x_0^3x_2\in \overline{M_{-1}}$ respectively correspond to the following aroma and tree with free edges:
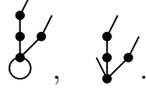
\[\onearomaFreeex, \quad \onearomarootedFreeex.\]
We emphasize that the monomials are enforced to be different as they belong to two different spaces.
\end{ex}

% The algebra of clumped multi-indices $(S(\AA M_{-1}), \diamond)$ is embedded into the wider symmetric algebra of clumped monomials $(\overline{\CC\MM}=S(\overline{\AA M_{-1}}),\diamond)$, and similarly for the clumped forests with free edges, denoted by $\overline{\CC\FF}$.
% The clumped forests with free edges will be denoted by $\overline{\CC\FF}=S(\overline{\AA T})$.
% The weight $\mathbf {wt}(\overline{a\tau})$ of $\overline{a\tau}\in \overline{\CC\FF} $ is given by the total number of the edges minus the number of vertices.
% The pairing extends straightforwardly to the $\overline{\CC\FF}$, where the symmetry factor of a forest with free edges is understood as the symmetry factor of the corresponding $\CC\times \mathbb N_0$-decorated forest.

% \modz{For the clumped multi-indices $(S(\AA M_{-1}),\diamond)$ is a  Hopf algebra. really Hopf algebra?} $(S(\overline{\AA M_{-1}}),\diamond)$ collects the elements with weight bigger than $-1$, which is analogous to the clumped forests with free edges. Compared with the clumped forest, it can be seen as the symmetric algebra generated by $\CC\times \mathbb N_0$. 

As $\partial\colon M_n\rightarrow M_{n+1}$, the derivation induces naturally the maps
\[
\partial\colon \overline{M_0}\rightarrow \overline{M_0},\quad
\partial\colon \overline{M_{-1}}\rightarrow \overline{M_{-1}}.
\]
The derivation $\partial$ is extended on $\overline{\AA M_{-1}}$ as a derivation:
\[
\partial(x^{\kappa^1}\odot \cdots \odot x^{\kappa^m} \odot x^{\mathbf k})
=\sum_{j=1}^m x^{\kappa^1}\odot \cdots \odot  \partial x^{\kappa^j} \odot \cdots \odot x^{\kappa^m}\odot x^{\mathbf k}
+ x^{\kappa^1}\odot \cdots \odot x^{\kappa^m} \odot \partial x^{\mathbf k}.
\]
Analogously to \eqref{partialderivation}, let the transpose derivation on $M_n$:
\[
    \overline{\partial} x^{\mathbf k}=\sum_{j\geq 0, a\in C} k_j^a x^{\mathbf k - \mathbf e_j^a + \mathbf e_{j-1}^a}.
\]
While $\partial$ increases the weight, $\overline{\partial}$ decreases it.
We extend $\overline{\partial}$ to $\overline{M_0}$, $\overline{M_{-1}}$, and $\overline{\AA M_{-1}}$ as before.

\begin{proposition}
The transpose derivation $\overline{\partial}$ is the transpose of the derivation $\partial$, that is, for $P,Q\in \overline{M_0}$ or $P,Q\in \overline{M_{-1}}$,
\[
\langle \partial P, Q\rangle=\langle P , \overline{\partial} Q \rangle.
\]
\end{proposition}

\begin{proof}
We already have the result on $\overline{M_{-1}}$ from \cite{Zhu24fna}.
Then, for $x^{\mu}, x^{\nu}\in \overline{M_0}$, we have
\begin{align*}
\langle \partial x^{\mu}, x^{\nu}\rangle & =  \sum_{t\geq 0, a\in C} \mu_i^a\langle x^{\mu - e_t^a + e_{t+1}^a}, x^{\nu} \rangle\\
& = \sum_{t\geq 0, a\in C} \mu_i^a \nu!\ind_{(\mu - e_t^a + e_{t+1}^a=\nu)}\\
& = \sum_{t\geq 0, a\in C}(\nu + e_t^a)!\ind_{(\mu - e_t^a + e_{t+1}^a=\nu)}.
\end{align*}
On the other hand, we find
\begin{align*}
\langle x^{\mu} , \overline{\partial} x^{\nu}\rangle & = \sum_{t\geq 1, a\in C}\nu_t^a\langle x^{\mu}, x^{\nu - e_t^a + e_{t-1}^a}\rangle\\
& = \sum_{t\geq 1, a\in C} \nu_t^a (\nu - e_t^a + e_{t-1}^a)!\ind_{(\nu - e_t^a + e_{t-1}^a = \mu)}\\
& = \sum_{t\geq 1, a\in C} (\nu + e_{t-1}^a)!\ind_{(\nu - e_t^a + e_{t-1}^a=\mu)}.
\end{align*}
Hence the result.
\end{proof}

An explicit expression for the iterated transpose derivation is given by the following result.
\begin{proposition}
\label{prop:formula_partial_bar}
For any multi-indices $\kappa, \mathbf k$, denote their corresponding left-shifted multi-indices by $\shift{\kappa},\shift{\mathbf k}$ with coordinates $\shift{\kappa_t^a}:=\kappa_{t+1}^a$.
% In particular, we have
% \[\shift{\mathbf e_t^a}=\mathbf e_{t-1}^a.\]
For any integer $r\geq 0$, the following formula holds
% \begin{equation}\label{dr}
\[
    \overline{\partial}^r( \mathbf y\odot  x^{\mathbf k})=\sum_{\vert\ell^1\vert +\cdots + \vert\ell^n\vert + \vert\ell\vert=r}C_{\kappa^1, \ell^1}\cdots C_{\kappa^n, \ell^n}C_{ k, \ell} x^{\kappa^1 - \ell^1 + \shift{\ell^1}}\odot \cdots\odot   x^{\kappa^n - \ell^n + \shift{\ell^n}}\odot  x^{\mathbf k - \ell + \shift{\ell}},
\]
% \end{equation}
where $\mathbf y=x^{\kappa_1}\odot \cdots \odot x^{\kappa_n}$ and the coefficient $C_{\kappa^i, \ell^i}, C_{\mathbf k, \ell}$ are given by $C_{\kappa^i, 0}=C_{\mathbf k, 0}=0$ for $1 \le i\le n$ and the recursive formula below:
\begin{align*}
    C_{\mathbf k, \ell} & = \sum_{t\geq -1, a\in C, \ell_t\geq 1}C_{\mathbf k, \ell -\mathbf e_t^a}(\mathbf k_t^a - \ell_t + 1 + \ell_{t+1}^a)\\
C_{\kappa^i, \ell^i} & = \sum_{t'\geq 0, a\in C, \ell^i_{t'}\geq 1} C_{ \kappa^i, \ell^i-\mathbf e_{t'}^a}(\kappa^{i,a}_{t'} - \ell^i_{t'} + 1 + \ell^{i,a}_{t'+1})
\end{align*}
\end{proposition}

\begin{proof}
The definition of $\overline{\partial}$ immediately yields
\[
C_{\mathbf k, \mathbf e_t^a} = \mathbf k_t^a,\quad
C_{\mathbf \kappa^i, \mathbf e_{t'}^a} = \kappa^{i,a}_{t'}.
\]
We therefore obtain
\begin{align*}
\overline{\partial}^{r+1}&( \mathbf  y\odot  x^{\mathbf k})
  =\overline{\partial}\hskip-6mm\sum_{\vert \ell^{1'}\vert + \cdots + \vert \ell^{n'}\vert + \vert \ell'\vert=r} C_{\kappa^1, \ell^{1'}}\cdots C_{\kappa^n, \ell^{n'}}C_{\mathbf k, \ell} x^{\kappa^1 - \ell^{1'} + \shift{\ell^{1'}}}\odot \cdots\odot x^{\kappa^n - \ell^{n'} + \shift{\ell^{n'}}}\odot x^{\mathbf k - \ell + \shift{\ell}}\\
 & =\hskip-6mm\sum_{\vert \ell^{1'}\vert + \cdots + \vert \ell^{n'}\vert + \vert \ell'\vert=r}\sum_{\ell^{i'} + \mathbf e_{t'}^a = \ell^i }C_{\kappa^1, \ell^{1'}} x^{\kappa^1 - \ell^1 + \shift{\ell^1}}\odot \cdots \odot  C_{\kappa^i, \ell^{i'} }(\kappa_{t'}^a + \ell_{t'+1}^{i,a} - \ell_{t'}^{i',a}) \\
 & \odot  x^{\kappa^i + \shift{(\ell^{i'}+ \mathbf e_{t'}^{a}) } - (\ell^{i'} + \mathbf e_{t'}^a)}\odot \cdots \odot C_{\kappa^n, \ell^n} x^{\kappa^n - \ell^n + \shift{\ell^n}}\odot    x^{\mathbf k - \ell + \shift{\ell}}\\
 & +\sum_{\vert \ell^{1'}\vert + \cdots + \vert \ell^{n'}\vert + \vert \ell'\vert=r}\sum_{\ell' + \mathbf e_t^a =\ell }C_{\kappa^1, \ell^1}\cdots C_{\kappa^n, \ell^n} x^{\kappa^1 - \ell^1 + \shift{\ell^1}}\odot  \cdots\odot  x^{\kappa^n - \ell^n + \shift{\ell^n}}\\
 & \odot C_{\mathbf k, \ell'}(k_t^a + \ell_{t+1}^a - \ell^{'a}_t)x^{\mathbf k + \shift{(\ell' + \mathbf e_t^a)} - (\ell' + \mathbf e_t^a)}\\
 & =\hskip -6mm \sum_{\vert \ell^{1}\vert + \cdots + \vert \ell^{n}\vert + \vert \ell\vert=r + 1}\sum_{(\ell^{i'}, t', a), \ell^{i'} + \mathbf e_{t'}^a=\ell^i}C_{\kappa^1, \ell^1} x^{\kappa^1 + \shift{\ell^1} - \ell^1}\odot \cdots\odot C_{\kappa^i, \ell^i - \mathbf e_{t'}^a}(\kappa_{t'}^a + \ell_{t'+1}^a - \ell_{t'}^a + 1) \\
 & \odot x^{\kappa^i + \shift{\ell^i} - \ell^i}\odot \cdots \odot C_{\kappa^n, \ell^n} x^{\kappa^n + \shift{\ell^n} - \ell^n}\odot C_{\mathbf k, \ell} x^{\mathbf k + \shift{\ell} - \ell}\\
 & + \hskip-6mm\sum_{\vert \ell^{1}\vert + \cdots + \vert \ell^{n}\vert + \vert \ell\vert=r + 1}\sum_{(\ell', t, a), \ell' + \mathbf e_t^a=\ell} C_{\kappa^1, \ell^1}\cdots C_{\kappa^n, \ell^n} x^{\kappa^1 - \ell^1 + \shift{\ell^1}}\odot \cdots\odot  x^{\kappa^n - \ell^n + \shift{\ell^n}}\\
 &   \odot C_{\mathbf k, \ell - \mathbf e_t^a}(\mathbf k_t^a + \ell_{t+1}^a - \ell_t^a + 1) x^{\mathbf k + \shift{\ell} - \ell}
\end{align*}
Hence the result.
\end{proof}

The following result is a universal combinatorial property applied in our context of multi-indices and tree-like structures.
\begin{proposition}\label{monomial and forest}
Let $\M=M_1\diamond \cdots \diamond M_r\in \CC\MM$ and $\F\in \CC\FF$ such that $\Phi(\mathbb F)=\mathbb M$.
% $(M_1,\cdots, M_r)$ be a $r$-tuple of monomials in $M_0$ or $M_{-1}$ and let the clumped monomial $\mathbb M=M_1\odot \cdots \odot M_r$ obtained by multiplying the $M_j$'s together. Let $\mathbb F$ be a clumped forest such that $\Phi(\mathbb F)=\mathbb M$ and let 
Let $\mathcal B$ be the set of $r$-tuples of aromatic trees given by
\[\mathcal B:= \{(\tau_1,\cdots,\tau_r), \tau_1\cdots \tau_r=\mathbb F \ and \ \Phi(M_j)=\tau_j \ for \ any \ j=1,\cdots,r\}.\]
Then, one finds
\[\vert \mathcal B\vert=\frac{\sigma^{ext}(\mathbb M)}{\sigma^{ext}(\mathbb F)}.\]
\end{proposition}

\begin{proof}
The external automorphism group of $\mathbb M$ acts transitively on $\mathcal B$. The stabiliser of the $r$-tuple $(t_1,\cdots, t_r)$ is the external automorphism group of the forest $F$. One concludes by the orbit-stabiliser theorem.
\end{proof}

Let $\delta: \overline{\CC\FF}\rightarrow \overline{\CC\FF}$ be given by
\[\delta(a\tau)= \sum_{v\in V(a\tau)}\delta_v(a\tau),\]
where $\delta_v$ adds a free edge to vertex $v$.
Analogously, let $\overline{\delta}$ be
\[\overline{\delta}(a\tau)= \sum_{v\in V(a\tau)}\overline{\delta}_v(a\tau).\]
where $\overline{\delta}_v$ removes a free edge at vertex $v$, and returns zero if $v$ has no free edge.

\begin{definition}
Let $\overline{\Phi}$ be the extension of the fertility map $\Phi$ on $\overline{\AA\FF}$ defined as
\begin{align*}
\Phi: \overline{\AA\FF} & \rightarrow \overline{\AA\MM}\\
\overline{a^1}\cdots \overline{a^n}\,\overline{t}& \mapsto \overline{\Phi}(\overline{a^1})\odot \cdots \odot \Phi(\overline{a^n})\odot \overline{\Phi}(\overline{t})
\end{align*}
where
\begin{align*}
\overline{\Phi}(\overline{t})=\prod_{v\in V(\overline{t})}x_{f(v) - 1}^{d(v)}, \quad \overline{\Phi}(\overline{a^j})=\prod_{v\in V(a^j)}x_{f(v)-1}^{d(v)}. 
\end{align*}
Analogously, formula \eqref{formulajaro} allows to extend the map $j^{aro}$ into
\begin{align*}
\overline{j^{aro}}: \overline{\AA\MM} \rightarrow \overline{\AA\FF},
\end{align*}
where the pairing is extended naturally to $\overline{\AA\FF}$, where the symmetry factor of aromatic forest with free edges is the symmetry factor of the corresponding $\CC\times \mathbb N_0$ decorated aromatic forest.
\end{definition}

\newcommand{\threearomafree}{
\tikz[planar forest ] {

\draw (-1.0,1.0) edge[-] (0.0,1.0);
\draw (0.0, 2.0) edge[-] (0.0, 1.0);
\draw (0.0,1.0) edge[-] (1.5,1.0);
\draw (-1.0,1.0) arc (90:270:0.5);
\draw (1.5,1.0) arc (90:-90:0.5);
\draw (-1.0, 0.0) edge[-] (1.5, 0.0);
\draw (1.5, 2.0) edge[-] (1.5, 1.0);
\draw (-1.0, 2.0) edge[-] (-1.0, 1.0);
\node [draw=none] at (0.0, 0.0) { } 
child {node [b] at (-1.0, 1.0) {  } edge from parent[draw=none] 
}
child {node [b] at (0.0, 1.0) {  } edge from parent[draw=none] 
}
child {node [b] at (1.5, 1.0) {  } edge from parent[draw=none] 
}
;
}}

\newcommand{\twoaromafree}{
\tikz[planar forest ] {

\draw (-1.0,1.0) edge[-] (0.0,1.0);
\draw (-1.0, 0.0) edge[-] (0.0, 0.0);
\draw (-0.5, 3.0) edge[-] (0.0, 2.0);
\draw (0.5, 3.0) edge[-] (0.0, 2.0);
\draw (-1.0,1.0) arc (90:270:0.5);
\draw (0.0,1.0) arc (90:-90:0.5);
\draw (-1.0, 2.0) edge[-] (-1.0, 1.0);
\node [draw=none] at (0.0, 0.0) { } 
child {node [b] at (-1.0, 1.0) {  } edge from parent[draw=none] 
}
child {node [b] at (0.0, 1.0) {  } edge from parent[draw=none] 
child {node [b] at (0.0, 1.0) {  } }
}
;
}}

\newcommand{\onearomafree}{
\tikz[planar forest ] {

\draw (0.0, 3.0) edge[-] (-1.0, 2.0);
\draw (-0.5, 4.0) edge[-] (-1.0, 3.0);
\draw (-1.5, 4.0) edge[-] (-1.0, 3.0);
\draw (-1.0,1.0) arc (90:270:0.5);
\draw (-1.0,1.0) arc (90:-90:0.5);
\node [draw=none] at (0.0, 0.0) { } 
child {node [b] at (-1.0, 1.0) {  } edge from parent[draw=none] 
child {node [b] at (0.0, 1.0) {}
child {node [b] at (0.0, 1.0){}}
}
}
;
}}

\newcommand{\threearomafreePhi}{
\tikz[planar forest ] {

\draw (-1.0,1.0) edge[-] (0.0,1.0);
\draw (0.0, 2.0) edge[-] (0.0, 1.0);
\draw (0.0,1.0) edge[-] (1.5,1.0);
\draw (-1.0,1.0) arc (90:270:0.5);
\draw (1.5,1.0) arc (90:-90:0.5);
\draw (-1.0, 0.0) edge[-] (1.5, 0.0);
\draw (-1.0, 2.0) edge[-] (-1.0, 1.0);
\draw (2.0, 3.0) edge[-] (1.5, 2.0);
\node [draw=none] at (0.0, 0.0) { } 
child {node [b] at (-1.0, 1.0) {  } edge from parent[draw=none] 
}
child {node [b] at (0.0, 1.0) {  } edge from parent[draw=none] 
}
child {node [b] at (1.5, 1.0) {  } edge from parent[draw=none] 
child {node [b] at (0.0, 1.0) {  }}
}
;
}}

\newcommand{\twoaromaFree}{
\tikz[planar forest ] {

\draw (-1.0,1.0) edge[-] (0.0,1.0);
\draw (-1.0, 0.0) edge[-] (0.0, 0.0);
\draw (0.5, 3.0) edge[-] (0.5, 2.0);
\draw (-1.0,1.0) arc (90:270:0.5);
\draw (0.0,1.0) arc (90:-90:0.5);
\draw (-0.5, 3.0) edge[-] (-0.5, 2.0);
\node [draw=none] at (0.0, 0.0) { } 
child {node [b] at (-1.0, 1.0) {  } edge from parent[draw=none] 
}
child {node [b] at (0.0, 1.0) {  } edge from parent[draw=none] 
child {node [b] at (-0.5, 1.0) {  } }
child {node [b] at (0.5, 1.0) {}}
}
;
}}

\newcommand{\threearomaFree}{
\tikz[planar forest ] {

\draw (-1.0,1.0) edge[-] (0.0,1.0);
\draw (0.0,1.0) edge[-] (1.5,1.0);
\draw (-1.0,1.0) arc (90:270:0.5);
\draw (1.5,1.0) arc (90:-90:0.5);
\draw (-1.0, 0.0) edge[-] (1.5, 0.0);
\draw (2.0, 3.0) edge[-] (1.5, 2.0);
\draw (2.0, 2.0) edge[-] (1.5, 1.0);
\node [draw=none] at (0.0, 0.0) { } 
child {node [b] at (-1.0, 1.0) {  } edge from parent[draw=none] 
}
child {node [b] at (0.0, 1.0) {  } edge from parent[draw=none] 
}
child {node [b] at (1.5, 1.0) {  } edge from parent[draw=none] 
child {node [b] at (0.0, 1.0) {} }
}
;
}}

\newcommand{\onearomaFree}{
\tikz[planar forest ] {

\draw (-0.5, 4.0) edge[-] (-1.0, 3.0);
\draw (0.5, 3.0) edge[-] (0.0, 2.0);
\draw (-1.0,1.0) arc (90:270:0.5);
\draw (-1.0,1.0) arc (90:-90:0.5);
\node [draw=none] at (0.0, 0.0) { } 
child {node [b] at (-1.0, 1.0) {  } edge from parent[draw=none] 
child {node [b] at (1.0, 1.0) {}}
child {node [b] at (0.0, 1.0) {}
child {node [b] at (0.0, 1.0){}}
}
}
;
}}

\newcommand{\aromatreefree}{
\tikz[planar forest ] {

\draw (-1.0,1.0) edge[-] (0.0,1.0);
\draw (0.0, 2.0) edge[-] (0.0, 1.0);
\draw (0.0,1.0) edge[-] (1.5,1.0);
\draw (-1.0,1.0) arc (90:270:0.5);
\draw (1.5,1.0) arc (90:-90:0.5);
\draw (-1.0,0.0) edge[-] (1.5,0.0);
\draw (1.0, 2.0) edge[-] (1.5, 1.0);
\draw (2.0, 2.0) edge[-] (1.5, 1.0);
\node [draw=none] at (0.0, 0.0) { } 
child {node [b] at (-1.0, 1.0) {  } edge from parent[draw=none] 
}
child {node [b] at (0.0, 1.0) {  } edge from parent[draw=none] 
}
child {node [b] at (1.5, 1.0) {  } edge from parent[draw=none] 
}
;
%\draw (-0.5, 1.5) -- (0.5, 1.5);
}}

\begin{ex}
One finds
\begin{align*}
\overline{j^{aro}}(x_1^3) & =6\,\threearomafree + 6\,\twoaromafree + 6\,\onearomafree, \quad \overline{j^{aro}}(x_0^3x_2)=3\, \twoaromaFree + 6\,\threearomaFree + 6\, \onearomaFree,\\
\overline{\Phi}(\threearomafree) & = x_1^3, \quad
\overline{\Phi}(\threearomafreePhi) = x_0x_1^3, \quad \overline{\Phi}(\threearomaFree)=x_0^3x_2, \quad \overline{\Phi}(\aromatreefree)=x_2x_1x_0. \\
\end{align*}
\end{ex}

\begin{proposition}
\label{prop:fertility_compatibility}
The map $\delta$ is the transpose of $\overline{\delta}$, that is, for $\tau^1,\tau^2\in \overline{A_0}$ or $\tau^1,\tau^2\in \overline{T}$,
\[
\langle \delta \tau^1, \tau^2\rangle=\langle \tau^1 , \overline{\delta} \tau^2 \rangle.
\]
Moreover, $\partial$ and $\delta$ satisfy the following identity on $\overline{A_0}$ and $\overline{T}$:
\[
\overline{\Phi}\circ \delta = \partial \circ \overline{\Phi}.
\]
\end{proposition}

\begin{proof}
Choosing a vertex $v$ (resp. $w$) in an aroma $a$ (resp. $b$), and considering the sets
\[P:= \{ v\in V(a), \delta_{v}(a)= b\}, \quad Q:= \{ w \in V(b), \overline{\delta_{w}}(b)=a\},\]
we find by the orbit-stabiliser theorem
\[ P \simeq Aut(a)/ Aut^v(a), \quad Q\simeq Aut(b)/ Aut^w(b).\]
where $v$ (resp. $w$) is any choice of element in $P$ (resp. $Q$), and where $Aut^v(a)$ is the stabiliser of $v$ in $Aut(a)$ (and similarly for $Q$). If two trees $a$ and $b$ are such that $b= \delta_v(a)$ for some $v\in V(a)$, then both sets of vertices $V(a)$ and $V(b)$ can be naturally identified. In that case, from the natural isomorphism $Aut^v(a)\simeq Aut^v(b)$, we obtain
\[\sigma(b)\vert P\vert = \sigma(a)\vert Q\vert.\]
We therefore have
\begin{align*}
    \langle \delta a, b\rangle
    &  = \vert P\vert \sigma(b)\\
& = \vert Q\vert \sigma(a)\\
& = \langle a, \overline{\delta} b\rangle.
\end{align*}
which proves the first assertion on $\overline{A_0}$. The proof is similar on $\overline{T}$.
\end{proof}

The transpose of the fertility identity in Proposition \ref{prop:fertility_compatibility} yields the following.
\begin{corollary}\label{deltabar-j}
The following identity holds on $\overline{M_0}$ and $\overline{M_{-1}}$:
\[
% \label{deltabar-j}
\overline{\delta}\circ \overline{j^{aro}}=\overline{j^{aro}}\circ\overline{\partial}.
\]
\end{corollary}

\begin{proposition}\label{delete-free-edges}
let $a\tau$ be an aroma tree with $r$ free edges. Let $r_v$ be the number of the free edges at the vertex $v$ of $a\tau$. Then we have
\[\overline{\delta}^r(a\tau)=\frac{r!}{\prod_{v\in \mathcal V(a\tau)}r_v!}a_0\tau_0,\]
where $a_0\tau_0$ is the aromatic tree with all free edges removed. For aromas only, one finds
\[\overline{\delta}^r(a)=\frac{r!}{\prod_{v\in \mathcal V(a)}r_v!}a_0.\]
\end{proposition}

\begin{proof}
We have 
\[\overline{\delta}^r(a\tau)= \sum_{(v_1,\ldots, v_r)\in \mathcal V(a\tau)^r}\overline{\delta}_{v_1}\circ \cdots \circ \overline{\delta}_{v_r}(a\tau).\]
A term in the right-hand side is equal to $a_0$ if and only if each vertex $v\in \mathcal V(a\tau)$ appears exactly $r_v$ times in the tuple $(v_1,\ldots, v_r)$, otherwise the term is equal to zero. The number of such $r$-tuples is equal to the multi-monomial coefficient above.
\end{proof}

For any admissible cut $c$ of an aromatic tree $a\tau\in \AA T$, let $\overline{R}_{c}(a\tau)$ be the associated full trunk, which is given by the trunk $R_c(a\tau)$ together with each cut edges replaced by a free edge. The weight of the full trunk is equal to the number of edges belonging to the cut.
In view of Proposition \ref{delete-free-edges}, we have
\begin{equation}
\label{fulltrunk}
\overline{\delta}^r\big(\overline{R}^c(a\tau)\big)=\|\overline{a}\, \overline{\tau}\|R^c(a\tau), \quad
\|\overline{a}\, \overline{\tau}\|:= \frac{r!}{\prod_{v\in \mathcal V(\overline{a}\, \overline{\tau})}r_v!},
\end{equation}
where $\overline{a}\, \overline{\tau}$ is the shorthand for $\overline{R}^c(a\tau)$.

\subsection{Combinatorics of the graftings and cuts}

\begin{definition}\label{defngrafting}
 Let $r$ be a positive integer, let $A$ be an aromatic forest with $r$ connected components without free edges, and let $\overline{a}$ be an aromatic tree with $r$ free edges. A grafting of $A$ on $\overline{a}$ consists in a a given way to attach the $r$ roots in $A$ to the $r$ free edges in $\overline{a}$, removing the free edges in the operation. The set of all possible graftings of $A$ on $\overline{a}$ is denoted $\mathcal G(A,\overline{a})$ and satisfies 
 \begin{equation}\label{numgrafting}
 \vert \mathcal G(A,\overline{a})\vert= \|\overline{a}\|.   
 \end{equation}
 For a vertex $v\in \mathcal V(\overline{a})$ and a grafting $b\in \mathcal G(A, \overline{a})$, we denote by $F_b(v)$ the subforest of $A$ attached to $v$ via $b$.
\end{definition}

\newcommand{\forestBCKexf}{
\tikz[planar forest ] {

\node [b] at (0.0, 0.0) {  } 
child {node [b] at (-0.5, 1.0) {  }  
}
child {node [b] at (0.5, 1.0) {  }  
child {node [b] at (0.0, 1.0) {  }  
}
}
;
\draw (-0.5, 0.5) -- (0.0, 0.5);
%\draw (0.0, 0.5) -- (0.5, 0.5);
\draw (0.25, 1.5) -- (0.75, 1.5);
}}

\newcommand{\aromatreecut}{
\tikz[planar forest ] {

\draw (-1.0,1.0) edge[-] (0.0,1.0);
\draw (0.0,1.0) edge[-] (1.5,1.0);
\draw (-1.0,1.0) arc (90:270:0.5);
\draw (1.5,1.0) arc (90:-90:0.5);
\draw (-1.0,0.0) edge[-] (1.5,0.0);
\node [draw=none] at (0.0, 0.0) { } 
child {node [b] at (-1.0, 1.0) {  } edge from parent[draw=none] 
}
child {node [b] at (0.0, 1.0) {  } edge from parent[draw=none] 
child {node [b] at (0.0, 1.0) {  }  
}
}
child {node [b] at (1.5, 1.0) {  } edge from parent[draw=none] 
child {node [b] at (-0.5, 1.0) {  }  
}
child {node [b] at (0.5, 1.0) {  }  
}
}
;
\draw (-0.25, 1.5) -- (0.25, 1.5);
\draw (1.0, 1.6) -- (1.5, 1.6);
\draw (1.5, 1.4) -- (2.0, 1.4);
}}

\newcommand{\forestFBcad}{
\tikz[planar forest ] {

\draw (-0.5, 1.5) edge[-] (0.5, 1.5);
\draw (0.0,1.0) arc (90:270:0.5);
\draw (0.0,1.0) arc (90:-90:0.5);
\node [draw=none] at (0.0, 0.0) { } 
child {node [b] at (0.0, 1.0) {  } edge from parent[draw=none] 
child {node [b] at (0.0, 1.0) {  }  
child {node [b] at (0.0, 1.0) {  }  
}
}
}
;
\node [b] at (1.1, 0.0) {  } 
;
}}

\newcommand{\forestYBacut}{
\tikz[planar forest ] {

\draw (0.0,1.0) arc (90:270:0.5);
\draw (0.0,1.0) arc (90:-90:0.5);
\node [draw=none] at (0.0, 0.0) { } 
child {node [b] at (0.0, 1.0) {  } edge from parent[draw=none] 
}
;
\draw (1.0, 1.5) edge[-] (1.5, 1.5);
\draw(2.0, 1.5) edge[-] (2.5, 1.5);
\draw (1.2,1.0) edge[-] (2.2,1.0);
\draw (1.2,1.0) arc (90:270:0.5);
\draw (2.2,1.0) arc (90:-90:0.5);
\draw (1.2,0.0) edge[-] (2.2,0.0);
\node [draw=none] at (1.7, 0.0) { } 
child {node [b] at (-0.5, 1.0) {  } edge from parent[draw=none] 
child {node [b] at (0.0, 1.0) {  }  
}
}
child {node [b] at (0.5000000000000002, 1.0) {  } edge from parent[draw=none] 
child {node [b] at (0.0, 1.0) {  }  
}
}
;
}}

\newcommand{\forestBCKexffree}{
\tikz[planar forest ] {

\coordinate (o) at (-0.5, 1.0);
\draw (-0.5, 1.0)--(0.0, 0.0);
\draw (0.0, 2.0)--(0.5, 1.0);
\node [b] at (0.0, 0.0) {  } 
child {node [b] at (0.5, 1.0) {  }  
};
%\draw (-0.5, 0.5) -- (0.0, 0.5);
%\draw (0.0, 0.5) -- (0.5, 0.5);
%\draw (0.25, 1.5) -- (0.75, 1.5);
}}

\newcommand{\aromatreefreebis}{
\tikz[planar forest ] {

\draw (-1.0,1.0) edge[-] (0.0,1.0);
\draw (0.0, 2.0) edge[-] (0.0, 1.0);
\draw (0.0,1.0) edge[-] (1.5,1.0);
\draw (-1.0,1.0) arc (90:270:0.5);
\draw (1.5,1.0) arc (90:-90:0.5);
\draw (-1.0,0.0) edge[-] (1.5,0.0);
\draw (1.0, 2.0) edge[-] (1.5, 1.0);
\draw (2.0, 2.0) edge[-] (1.5, 1.0);
\node [draw=none] at (0.0, 0.0) { } 
child {node [b] at (-1.0, 1.0) {  } edge from parent[draw=none] 
}
child {node [b] at (0.0, 1.0) {  } edge from parent[draw=none] 
}
child {node [b] at (1.5, 1.0) {  } edge from parent[draw=none] 
}
;
%\draw (-0.5, 1.5) -- (0.5, 1.5);
}}

\newcommand{\forestFBfree}{
\tikz[planar forest ] {

\draw (0.0, 2.0) edge[-] (0.0, 1.0);
\draw (0.0,1.0) arc (90:270:0.5);
\draw (0.0,1.0) arc (90:-90:0.5);
\node [draw=none] at (0.0, 0.0) { } 
child {node [b] at (0.0, 1.0) {  } edge from parent[draw=none] 
}
;
\node [b] at (1.1, 0.0) {  } 
;
}}

\newcommand{\forestYBfree}{
\tikz[planar forest ] {

\draw (0.0,1.0) arc (90:270:0.5);
\draw (0.0,1.0) arc (90:-90:0.5);
\node [draw=none] at (0.0, 0.0) { } 
child {node [b] at (0.0, 1.0) {  } edge from parent[draw=none] 
}
;

\draw (1.2, 2.0) edge[-] (1.2, 1.0);
\draw (2.2000000000000002, 2.0) edge[-] (2.2000000000000002, 1.0);
\draw (1.2,1.0) edge[-] (2.2,1.0);
\draw (1.2,1.0) arc (90:270:0.5);
\draw (2.2,1.0) arc (90:-90:0.5);
\draw (1.2,0.0) edge[-] (2.2,0.0);
\node [draw=none] at (1.7, 0.0) { } 
child {node [b] at (-0.5, 1.0) {  } edge from parent[draw=none] 
}
child {node [b] at (0.5000000000000002, 1.0) {  } edge from parent[draw=none] 
}
;
}}

Examples of admissible cuts are presented below. It should be noted that an admissible cut for aroma cannot cut loops:
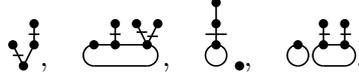
\[
\forestBCKexf, \quad \aromatreecut, \quad \forestFBcad, \quad \forestYBacut.
\]
The associated main component with free edges are
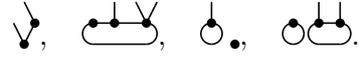
\[
\forestBCKexffree, \quad \aromatreefreebis, \quad \forestFBfree, \quad \forestYBfree.
\]

For the choices of admissible cuts and graftings on rooted part, we have the following result, adapted from \cite{Zhu24fna} for aromas.
\begin{proposition}\label{propgraftcut}
Let $r$ be a positive integer, let $A$ and $\overline{a}$ as in Definition \ref{defngrafting} and let $at$ be an aromatic tree without free edges. Let $\mathcal G(a, A,\overline{a})$ be the set of graftings of $A$ on $\overline{a}$ resulting in the tree $a$. Let $\mathcal C(a, A, \overline{a})$ be the set of admissible cuts of $a$ such that $P^c(a)= A$ and $\overline{R}^c(a)=\overline{a}$. Then, one has
\[\vert \mathcal C(a, A, \overline{a})\vert=\frac{\sigma(a)}{\sigma(A)\sigma(\overline{a})}\vert \mathcal G(a, A,\overline{a})\vert.\]
\end{proposition}

\begin{proof}
The group $\mop{Aut} a$ acts transitively on $\mathcal C(a, A, \overline{a})$. The stabiliser of a cut $c$ will be denoted by $\mop{Aut}_{c} a \tau$ and its cardinal by $\sigma_c(a)$. On the other hand, the group $\mop{Aut }\overline{a}\times \mop{Aut} A$ acts transitively on $\mathcal G(a, A, \overline{a})$. To see this, consider two graftings $b, b'\in \mathcal G(a, A, \overline{a})$: there exists a permutation $\alpha$ of $\mathcal V(\overline{a})$ such that $F_{b}(x)$ and $F_{b'}\big(\alpha(x)\big)$ are isomorphic, and the permutation $\alpha$ necessarily comes from an automorphism of the tree $a$. The stabiliser of $b$ is $\mop{Aut}_b \overline{a}\times \mop{Aut}_b A$, where $\mop{Aut}_b\overline{a}$ is the subgroup of those $\alpha\in \mop{Aut}\overline{a}$ such that $F_b\big(\alpha(x)\big)$ and $F_b(x)$ are isomorphic for any vertex $x$ of $\overline{a}$, and where $\mop{Aut}_b A=\prod_{x\in \mathcal {V}(\overline{a})}\mop{Aut}_b(x)$ is the subgroup $\mop{Aut} A$ which respects the subforests $F_b(x)$. By the orbit-stabiliser theorem, we therefore have
\[\frac{\vert \mathcal C(a, A, \overline{a})\vert }{\vert \mathcal G(a, A, \overline{a})\vert}=\frac{\sigma(a)}{\vert\mop{Aut}_c(a)\vert}\frac{\vert\mop{Aut}_b (\overline{a})\vert.\vert \mop{Aut}_b (A)\vert }{\sigma(A)\sigma(\overline{a})}\]
We conclude by noticing the natural isomorphism $\mop{Aut}_c(a) \sim \mop{Aut}_b (\overline{a})\times \mop{Aut}_b (A)$.
\end{proof}

\subsection{The aromatic LOT coproduct}\label{explicit formula}

Let us provide an alternative definition of the aromatic LOT coproduct and show that is it equivalent to the one given in Theorem \ref{arobck1}.
\begin{definition}\label{admcutmonomial}
An admissible cut $\mathbf c$ of a monomial $x^{\kappa}\in M_0$ is a tuple of the form $(\mathbf k^1,\dots, \mathbf k^r,\overline{\kappa})$ that satisfies
\[\kappa=\mathbf k^1+\cdots+\mathbf k^r+\overline{\kappa},
\quad
\mathbf{wt}(x^{\mathbf k^j})=-1.
\]
Analogously, an admissible cut $\mathbf c$ of a monomial $x^{\mathbf k}\in M_{-1}$ is a tuple of the form $(\mathbf k^1,\dots, \mathbf k^r,\overline{\mathbf k})$ that satisfies
\[\mathbf k=\mathbf k^1+\cdots+\mathbf k^r+\overline{\mathbf k},
\quad
\mathbf{wt}(x^{\mathbf k^j})=-1.
\]
We associate to an admissible cut on $M_0$ the following maps (and analogously on $M_{-1}$),
\[
P^{\mathbf c}(x^{\kappa})=x^{\mathbf k^1}\odot \cdots \odot x^{\mathbf k^r},\quad
R^{\mathbf c}(x^{\kappa})=\overline{\partial}^r x^{\overline{\kappa}}.
\]
The LOT coproduct is given on $M_0$ and $M_{-1}$ by
\[\Delta_{LOT}^{aro}(x^{\kappa})=x^{\kappa}\otimes \mathbf{1}+\sum_{\mathbf c \in Adm(x^{\kappa})} P^{\mathbf c}(x^{\kappa})\otimes R^{\mathbf c}(x^{\kappa}),\quad
\Delta_{LOT}^{aro}(x^{\mathbf k})=\sum_{\mathbf c \in Adm(x^{\mathbf k})} P^{\mathbf c}(x^{\mathbf k})\otimes R^{\mathbf c}(x^{\mathbf k}).\]
Then, $\Delta_{LOT}^{aro}: \AA M_{-1} \rightarrow \AA\MM\otimes \AA M_{-1}$ is defined by \eqref{coproduct}.
\end{definition}

\begin{ex}
One finds for the monomial $x_{-1}^2x_{1}^2$ that
\begin{align*}
\Delta_{LOT}^{aro}(x_{-1}^2x_{1}^2)  =& x_{-1}^2x_{1}^2\otimes \mathbf 1 + \mathbf 1\otimes x_{-1}^2x_{1}^2 + 2x_{-1}\otimes \overline{\partial}(x_{-1}x_{1}^2) + 2x_{-1}^2x_{1}\otimes \overline{\partial}x_{1}\\
&   + x_{-1}\cdot x_{-1}\otimes \overline{\partial}^2(x_{1}^2)\\
 = & x_{-1}^2x_{1}^2 \otimes \mathbf 1 + \mathbf 1 \otimes x_{-1}^2x_{1}^2 + 4x_{-1}\otimes x_{-1}x_{0}x_{1} + 2x_{-1}^2x_{1}\otimes x_{0}\\
&   + 2x_{-1}\cdot x_{-1}\otimes x_{-1}x_{1} + 2x_{-1}\cdot x_{-1}\otimes x_{0}^2
\end{align*}
\end{ex}

Let us denote by $\vert \mathbf c \vert$ the set of admissible cuts $\mathbf c'$ such that $P^{\mathbf c}(x^{\kappa}\odot  x^{\mathbf k})=P^{\mathbf c'}(x^{\kappa} \odot x^{\mathbf k})$ (and therefore $R^{\mathbf c}(x^{\kappa}\odot  x^{\mathbf k})=R^{\mathbf c'}(x^{\kappa}\odot x^{\mathbf k})$), and by $\| \mathbf c \|$ the cardinal of this class. We have 
\begin{eqnarray*}
\mathbf \|\mathbf c\| & = &\frac{\kappa!}{\overline{\kappa}!\sigma(x^{\kappa^1}\odot \cdots\odot x^{\kappa^r})}\cdot \frac{\mathbf k!}{\overline{\mathbf k!}\sigma(\mathbf x^{\mathbf k^1}\odot \cdots\odot x^{\mathbf k^t})}\\
& = & \frac{\kappa!}{\overline{\kappa}\kappa^1!\cdots \kappa^r!\sigma^{ext}(x^{\kappa^1}\odot\cdots\odot x^{\kappa^r})}\cdot \frac{\mathbf k!}{\overline{\mathbf k}\mathbf k^1!\cdots \mathbf k^t!\sigma^{ext}(x^{\mathbf k^1}\odot \cdots\odot  x^{\mathbf k^t})}
\end{eqnarray*}
whenever $P^{\mathbf c}(x^{\kappa}\odot x^{\mathbf k})=x^{\kappa^1}\odot \cdots\odot  x^{\kappa^r}\odot x^{\mathbf k^1}\odot \cdots\odot x^{\mathbf k^t}$, where $\sigma^{ext}$ is the external symmetry factor. In analogy with trees, the full trunk is defined by 
\[\overline{R}^{\mathbf c}(x^{\kappa}\odot x^{\mathbf k}):= x^{\overline{\kappa}}\odot x^{\overline{\mathbf k}}.\]

\begin{definition}\label{admcut}
 Let $\mathbf c$ be an admissible cut of the monomial $x^{\kappa}\odot x^{\mathbf k}$, and let $c$ be an admissible cut of aroma tree $at$. We say that $c$ matches $\mathbf c$ and write $c \sim \mathbf c$ whenever
 \begin{itemize}
 \item $x^{\kappa}\odot x^{\mathbf k}=\Phi(a)\Phi(\tau),$
 \item $P^{\mathbf c}(x^{\kappa }\odot x^{\mathbf k})= \Phi\big( P^{c}(a\tau)\big)$
 \end{itemize}
 An admissible cut $c$ matches $\mathbf c$ if and only if it matches any element $\mathbf c' \in \vert c \vert$. 
\end{definition}

Note that the second condition implies $\overline{R}^{\mathbf c}(x^{\kappa}\odot x^{\mathbf k})=\Phi\big( \overline{R}^{c}(a\tau)\big)$. Therefore, we have $P^{\mathbf c}(x^{\kappa}\odot x^{\mathbf k})=\Phi(P^c(a\tau))=\Phi(P^c(a))\Phi(P^c(\tau))= P^c(x^{\kappa})\odot P^c(x^{\mathbf k})$, and analogously for $\overline{R}^{\mathbf c}(x^{\kappa}\odot x^{\mathbf k})$. Then, Theorem \ref{arobck1} is a straightforward consequence of the following lemma.
\begin{lemma}\label{aroj}
For any monomial $x^{\kappa}\odot x^{\mathbf k}$ in $\mathcal \HH_{LOT}^{aro}$ and for any admissible cut $\mathbf c$ of $x^{\kappa}\odot x^{\mathbf k}$, the following holds:
\begin{equation}\label{cojembedding}
\| \mathbf c\|(j\otimes j)\big( P^{\mathbf c}(x^{\kappa}\odot x^{\mathbf k})\otimes R^{\mathbf c}(x^{\kappa}\odot x^{\mathbf k})\big)=\sum_{a\tau, \Phi(a\tau)=x^{\kappa}\odot x^{\mathbf k}}\sum_{c\in Adm(a\tau), c \sim \mathbf c}\frac{\sigma(x^{\kappa}\odot x^{\mathbf k})}{\sigma(a\tau)}P^{c}(a\tau)\otimes R^{c}(a\tau).
\end{equation}
\end{lemma}

\begin{proof}
Fix an admissible cut $\mathbf c$ of the monomial $x^{\kappa}$. Denoting the left-hand side and the right-hand side of \eqref{cojembedding} by $\mathcal L$ and $\mathcal R$ respectively.
A computation yields
\begin{align*}
\mathcal L & = \frac{\kappa!}{\kappa^1!\cdots \kappa^r!\overline{\kappa}!\sigma^{ext}(x^{\kappa^1}\odot \cdots\odot x^{\kappa^r})}j(x^{\kappa^1})\cdots j(x^{\kappa^r})\otimes j\circ \overline{\partial}^r(x^{\overline{\kappa}})\\
& = (\Id \otimes \overline{\delta}^r)\Big(\frac{\kappa!}{\kappa^1!\cdots \kappa^r!\overline{\kappa}!\sigma^{ext}(x^{\kappa^1} \odot \cdots\odot x^{\kappa^r})}j(x^{\kappa^1})\cdots j(x^{\kappa^r})\otimes j(x^{\overline{\kappa}})\Big)  \\
& (\text{from Corollary } \ref{deltabar-j})\\
& = \frac{\kappa!}{\kappa^1!\cdots \kappa^r!\overline{\kappa}!\sigma^{ext}(x^{\kappa^1}\odot\cdots\odot x^{\kappa^r})}(\Id \otimes \overline{\delta}^r)\left(\sum_{ \genfrac{}{}{0pt}{2}{(a^j)_{1,\ldots, r,}}  {\Phi(a^j)=\mathbf x^{\mathbf \kappa^j} }}\sum_{\overline{a}, \Phi(\overline{a})=x^{\overline{\kappa}}}\frac{\kappa^1!}{\sigma(a^1)} \cdots \frac{\kappa^r!}{\sigma(a^r)}\frac{\overline{\kappa}}{\sigma(\overline{a})}\right)\\
& =  \kappa!(\Id \otimes \overline{\delta}^r)\left(\sum_{A, \Phi(A)= x^{\kappa^1}\odot \cdots\odot x^{\kappa^r}}\sum_{\overline{a}, \Phi(\overline{a})=x^{\overline{\kappa}}}\frac{1}{\sigma(A)\sigma(\overline{a})}A\otimes \overline{a}\right).
\end{align*}
From equations \eqref{fulltrunk} and \eqref{numgrafting} and Proposition \ref{propgraftcut}, we obtain 
\begin{align*}
\mathcal L & = \kappa! \sum_{A, \Phi(A)=x^{\kappa^1}\odot \cdots\odot  x^{\kappa^r}}\sum_{\overline{a}, \Phi(\overline{a})=x^{\overline{\kappa}}}\| \overline{a}\|\frac{1}{\sigma(A)\sigma(\overline{a})}A\otimes \overline{a_0}\\
& = \kappa! \sum_{A, \Phi(A)= x^{\kappa^1}\odot \cdots\odot x^{\kappa^r}}\sum_{\overline{a}, \Phi(\overline{a})=x^{\overline{\kappa}}}\frac{\vert\mathcal G(A, \overline{a})\vert}{\sigma(A)\sigma(\overline{a})}A\otimes \overline{a}_0\\
& = \kappa! \sum_{a, \Phi(a)=x^{\kappa}}\sum_{A, \Phi(A)=x^{\kappa^1}\odot \cdots\odot x^{\kappa^r}}\sum_{\overline{a}, \Phi(\overline{a})=x^{\overline{\kappa}}}\frac{\vert \mathcal G(a, A, \overline{a})\vert }{\sigma(A)\sigma(\overline{a})}A\otimes \overline{a}_0\\
& = \kappa!\sum_{a, \Phi(a)=x^{\kappa}}
\sum_{A, \Phi(A)=x^{\kappa^1}\odot \cdots\odot  x^{\kappa^r}}\sum_{\overline{a}, \Phi(\overline{a})=x^{\overline{\kappa}}}\frac{\vert \mathcal C(a, A, \overline{a})\vert}{\sigma(a)}A\otimes \overline{a}_0\\
& = \sum_{a, \Phi(a)=x^{\kappa}}\frac{\kappa!}{\sigma(a)}\sum_{c\in Adm(a), c\sim \mathbf c}P^{c}(a)\otimes R^{c}(a)\\
& = \mathcal R.
\end{align*}
We obtain the expression of $\Delta_{LOT}^{aro}$ on $M_0$ and $M_{-1}$. Then, as the coproduct $\Delta_{LOT}^{aro}$ is a morphism for $\odot$, we have
\begin{align*}
(j\otimes j)\big(\Delta_{LOT}^{aro}(x^\kappa \odot x^{\mathbf k})\big) & = (j\otimes j)\big(\Delta_{LOT}^{aro}(x^\kappa)\odot \Delta_{LOT}^{aro}(x^{\mathbf k})\big)\\
& = \Delta_{LOT}^{aro}\big(j(x^\kappa)\big)\odot \Delta_{LOT}^{aro}\big(j(x^{\mathbf k})\big).
\end{align*}
This yields the identity \eqref{cojembedding}.
\end{proof}

Let us now study the links between $j^{cl}$ and $j^{aro}$ to prove Theorem \ref{thm:embedding}.
\begin{lemma}
\label{lemma:compatibility_j_maps}
The maps $j^{cl}$ and $j^{aro}$ satisfy the identity
\[j^{cl}\circ\varphi^*=\psi^* \circ j^{aro}.\]
\end{lemma}

\begin{proof}
We choose any element $\mathbf y \odot \mathbf z = x^{\kappa^1}\odot \dots \odot x^{\kappa^m}\odot x^{\mathbf k^1}\odot \dots \odot x^{\mathbf k^n}\in \HH_{LOT}^{aro}$, it can also be written as $\mathbf y \odot \mathbf z= \mathbf y^1 \odot \dots \odot \mathbf y^n \odot x^{\mathbf k^1}\odot \dots \odot x^{\mathbf k^n}$, here we should note that $\mathbf y=\mathbf y^1\odot \dots \odot \mathbf y^n$ and $\mathbf y^i, 1\leq i\leq n$ which satisfies $\Phi(\mathbf r^i)=\mathbf y^i$, where we write multi-aromas by $\mathbf r^i$ and aromas by $r^i$.
Then, a computation yields
\begin{align*}
\psi^*&\circ j^{aro}(\mathbf y \odot \mathbf z)
= \psi^*\circ j^{aro}(x^{\kappa^1}\odot \cdots \odot x^{\kappa^m}\odot x^{\mathbf k^1}\odot \cdots \odot x^{\mathbf k^n})\\
&= \psi^* \big((\sum_{\Phi(x^{\kappa^1})=r^1}\frac{\sigma(x^{\kappa^1})}{\sigma(r^1)}r^1)\cdots (\sum_{\Phi(x^{\kappa^m})=r^m}\frac{\sigma(x^{\kappa^m})}{\sigma(r^m)}r^m) \\
& \cdot (\sum_{\Phi(\tau^1)=x^{\mathbf k^1}}\frac{\sigma(x^{\mathbf k^1})}{\sigma(\tau^1)}\tau^1) \cdots (\sum_{\Phi(\tau^1)=x^{\mathbf k^n}}\frac{\sigma(x^{\mathbf k^n})}{\sigma(\tau^n)}\tau^n)\big)\\
&=\psi^* \big(\sum_{\genfrac{}{}{0pt}{2}{(r^i)_{1,\ldots, m, }(\tau^j)_{1, \ldots, n}}  {\Phi(r^i)=\mathbf x^{\mathbf \kappa^i}, \Phi(\tau^j)=x^{\mathbf k^j}}}\frac{\sigma(x^{\kappa^1})}{\sigma(r^1)}r^1\cdots \frac{\sigma(x^{\kappa^m})}{\sigma(r^m)}r^m\cdot \frac{\sigma(x^{\mathbf k^1})}{\sigma(\tau^1)}\tau^1\cdots \frac{\sigma(x^{\mathbf k^n})}{\sigma(\tau^n)}\tau^n\big)\\
&= \psi^* \big(\sum_{\genfrac{}{}{0pt}{2}{(\mathbf \mathbf r^i)_{1,\ldots, n, }(\tau^j)_{1, \ldots, n}}  {\Phi(\mathbf r^i)=\mathbf y^i, \Phi(\tau^j)=x^{\mathbf k^j}}}\frac{\sigma^{ext}(\mathbf y^1)\sigma^{int}(\mathbf y^1)}{\sigma^{ext}(\mathbf r^1)\sigma^{int}(\mathbf r^1)}\mathbf r^1\cdots \frac{\sigma^{ext}(\mathbf y^n)\sigma^{int}(\mathbf y^n)}{\sigma^{ext}(\mathbf r^n)\sigma^{int}(\mathbf r^n)}\mathbf r^n\cdot \frac{\sigma(x^{\mathbf k^1})}{\sigma(\tau^1)}\tau^1\cdots \frac{\sigma(x^{\mathbf k^n})}{\sigma(\tau^n)}\tau^n\big)\\
% &   & \text{(by proposition \ref{monomial and forest})} \\
& =\sum_{\genfrac{}{}{0pt}{2}{(\mathbf r^i)_{1,\ldots, n }}{\Phi(\mathbf r^1)\cdots \Phi(\mathbf r^n)=\mathbf y^1\odot \cdots \odot \mathbf y^n}} \sum_{\genfrac{}{}{0pt}{2}{(\mathbf r^i)_{1,\ldots, n, }(\tau^j)_{1, \ldots, n}}  {\Phi(\mathbf r^i \tau^j)=\mathbf y^i\odot x^{\mathbf k^j}}}(\frac{\sigma(\mathbf y^1 \odot x^{\mathbf k^1})}{\sigma(\mathbf r^1\tau^1)}\mathbf r^1\tau^1) \cdots (\frac{\sigma(\mathbf y^n \odot x^{\mathbf k^n})}{\sigma(\mathbf r^n\tau^n)}\mathbf r^n\tau^n)\\
& = \sum_{\genfrac{}{}{0pt}{2}{(\mathbf r^i)_{1,\ldots, n, }(\tau^j)_{1, \ldots, n}}  {\Phi(\mathbf r^i \tau^j)=\mathbf y^i\odot x^{\mathbf k^j}}} \sum_{\genfrac{}{}{0pt}{2}{(\mathbf r^i)_{1,\ldots, n }}{\Phi(\mathbf r^1)\cdots \Phi(\mathbf r^n)=\mathbf y^1\odot \cdots \odot \mathbf y^n}}(\frac{\sigma(\mathbf y^1 \odot x^{\mathbf k^1})}{\sigma(\mathbf r^1\tau^1)}\mathbf r^1\tau^1) \cdots (\frac{\sigma(\mathbf y^n \odot x^{\mathbf k^n})}{\sigma(\mathbf r^n\tau^n)}\mathbf r^n\tau^n)\\
& = j^{cl}\big(\sum_{\genfrac{}{}{0pt}{2}{(\mathbf r^i)_{1,\ldots, n }}{\Phi(\mathbf r^1)\cdots \Phi(\mathbf r^n)=\mathbf y^1\odot \cdots \odot \mathbf y^n}}(\mathbf y^1\odot x^{\mathbf k^1})\diamond \cdots \diamond(\mathbf y^n\odot x^{\mathbf k^n})\big)\\
& = j^{cl}\circ \varphi^*(\mathbf y^1\odot \cdots \odot \mathbf y^n\odot x^{\mathbf k^1}\odot \cdots \odot x^{\mathbf k^n})\\
& = j^{cl}\circ \varphi^*(\mathbf y \odot \mathbf z)
\end{align*}
Hence the result.
\end{proof}

\begin{proof}[Proof of the embeddings of Theorem \ref{thm:embedding}]
By applying Lemma \ref{aroj} and summing over all classes $\vert \mathbf c\vert$ of admissible cuts $\mathbf c$ of the monomial $x^{\kappa}\odot x^{\mathbf k}$, the coproduct $\Delta_{LOT}^{aro}$ verifies on $\AA \MM$:
\[
    (j^{aro}\otimes j^{aro})\circ \Delta_{LOT}^{aro}=\Delta_{BCK}^{aro}\circ j^{aro}.
\]
This implies that on $\AA M_{-1}$, we have
\begin{align*}
(j^{cl}\otimes j^{cl})\circ \Delta_{LOT}^{cl}
&=(j^{cl}\otimes j^{cl})\circ (\varphi^* \otimes \ind)\circ \Delta_{LOT}^{aro}\\
&=(\psi^*\otimes \ind) \circ (j^{aro}\otimes j^{cl})\circ \Delta_{LOT}^{aro}\\
&=(\psi^*\otimes \ind) \circ (j^{aro}\otimes j^{aro})\circ \Delta_{LOT}^{aro}\\
&=(\psi^*\otimes \ind)\circ \Delta_{BCK}^{aro}\circ j^{aro}\\
&=\Delta_{BCK}^{cl}\circ j^{cl}
\end{align*}
where we used that $j^{aro}$ and $j^{cl}$ coincide on $\AA M_{-1}$ and Lemma \ref{lemma:compatibility_j_maps}.
\end{proof}

%--------------------------------------------------------------
\bigskip

\noindent \textbf{Acknowledgements.}\
The authors would like to thank Dominique Manchon for fruitful discussions.
Adrien Busnot Laurent acknowledges the support from the program ANR-25-CE40-2862-01 (MaStoC - Manifolds and Stochastic Computations) and ANR-11-LABX-0020 (Labex Lebesgue).

\bibliographystyle{abbrv}
%\nocite{*}
\bibliography{ma_bibliographie}

\end{document}